\documentclass[ejs]{imsart}

\usepackage{epsfig}
\usepackage{amssymb,amsmath,amsthm,url}
\usepackage{multirow}
\usepackage{booktabs}

\usepackage{hyperref}
\usepackage{dsfont}
\usepackage{amsfonts,eucal}
\usepackage{graphicx}
\usepackage{color}
\usepackage{enumitem}
\usepackage{stackengine}
\usepackage{mathtools}

\RequirePackage[numbers]{natbib}
\RequirePackage[colorlinks,citecolor=blue,urlcolor=blue]{hyperref}

\doi{10.1214/154957804100000000}
\pubyear{0000}
\volume{0}
\firstpage{1}
\lastpage{0}
\arxiv{2005.07557}

\startlocaldefs
\theoremstyle{plain}
\newtheorem{theorem}{Theorem}[section]

\newtheorem{lemma}[theorem]{Lemma}
\newtheorem{proposition}[theorem]{Proposition}

\theoremstyle{definition}

\theoremstyle{remark}

\numberwithin{equation}{section}

\newcommand{\N}{\mathbb{N}}
\newcommand{\R}{\mathbb{R}}

\newcommand{\mX}{\mathbb{X}}
\newcommand{\mY}{\mathbb{Y}}
\newcommand{\mZ}{\mathbb{Z}}
\newcommand{\Z}{\mathbb{Z}}

\newcommand{\p}{\mathbb{P}}

\newcommand{\cC}{\mathcal{C}}

\newcommand{\cG}{\mathcal{G}}
\newcommand{\cF}{\mathcal{F}}

\newcommand{\cK}{\mathcal{K}}
\newcommand{\cL}{\mathcal{L}}

\newcommand{\cN}{\mathcal{N}}

\newcommand{\cP}{\mathcal{P}}
\newcommand{\cR}{\mathcal{R}}
\newcommand{\cQ}{\mathcal{Q}}

\newcommand{\fD}{\mathfrak{D}}

\newcommand{\fG}{\mathfrak{G}}

\newcommand{\fI}{\mathfrak{I}}
\newcommand{\fJ}{\mathfrak{J}}

\newcommand{\E}[1]{\mathbb{E}\left [  #1 \right ]}
\newcommand{\Ec}[1]{\mathbb{E}^*\left [  #1 \right ]}
\newcommand{\V}{\operatorname{Var}}

\renewcommand{\epsilon}{\varepsilon}
\renewcommand{\phi}{\varphi}

\newcommand{\intd}[1]{\mathrm{d}#1}

\newcommand{\1}[1]{\,\mathds{1}\! \left\{ #1 \right\} }
\newcommand{\cov}{\operatorname{Cov} }

\newcommand{\poi}{\operatorname{Poi}}
\newcommand{\diam}{\operatorname{diam}}

\newcommand{\ol}{\overline}
\newcommand{\wt}{\widetilde}
\endlocaldefs

\begin{document}

\allowdisplaybreaks

\begin{frontmatter}

\title{On approximation theorems for the Euler characteristic with applications to the bootstrap
}
\runtitle{Approximation theorems and bootstrap for the EC}


\author{\fnms{Johannes} \snm{Krebs}\corref{}
\ead[label=e1]{krebs@uni-heidelberg.de}}
\address{Institute of Applied Mathematics, Heidelberg University\\ 69120 Heidelberg, Germany\\ \printead{e1}}

\author{\fnms{Benjamin} \snm{Roycraft}
\ead[label=e2]{btroycraft@ucdavis.edu}}
\address{Department of Statistics, University of California\\ Davis, CA, 95616, USA\\ \printead{e2}}

\author{\fnms{Wolfgang} \snm{Polonik}
\ead[label=e3]{wpolonik@ucdavis.edu}}
\address{Department of Statistics, University of California\\ Davis, CA, 95616, USA\\ \printead{e3}}

\runauthor{Krebs, Roycraft, Polonik}

\begin{abstract}
We study approximation theorems for the Euler characteristic of the Vietoris-Rips and 
\v Cech filtration. The filtration is obtained from a Poisson or binomial sampling scheme in the critical regime. We apply our results to the smooth bootstrap of the Euler characteristic and determine its rate of convergence in the Kantorovich-Wasserstein distance and in the Kolmogorov distance.
\end{abstract}





\end{frontmatter}

\section{Introduction}
In this manuscript we study approximation results and central limit theorems for the Euler characteristic (EC) $\chi$ of a simplicial complex $\cK$. The EC is a simple yet major functional in topological data analysis (TDA). Recent contributions concerning the EC in TDA include Adler \cite{adler2008some}, Turner et al. \cite{turner2014}, Decreusefond et al. \cite{decreusefond2014simplicial} and Crawford et al. \cite{crawford2016functional}. See also \cite{Li2018plantmorph,Amezquita2020TDAbiol,Wilding2021cosmologies}. Multivariate central limit theorems for the EC were proved in Hug et al. \cite{hug2016second}. Ergodic theorems for the EC are given in Schneider and Weil \cite{schneider2008stochastic}. Thomas and Owada \cite{thomas2019functional} derive a functional strong law of large numbers and a functional central limit theorem (FCLT) for the EC obtained from the Vietoris-Rips complex of a Poisson process in the critical regime.

Among others, our results extend the findings of this last work to the \v Cech filtration and a binomial sampling scheme. 
More precisely, we obtain a FCLT and derive rates of convergence in the Kantorovich-Wasserstein distance and in the Kolmogorov distance at a fixed time parameter $t$ of the empirical process
\begin{align}\label{E:EmpProcess}
		\ol \chi_n \colon		[0,T] \to \R,  t \mapsto n^{-1/2} \big( \chi( \cK_{t,n} ) - \E{ \chi( \cK_{t,n} )} \big),
\end{align}
where $T\in\R_+$ and where $\cK_{t,n} = \cK_t(n^{1/d} \mX_n)$ for certain point clouds $\mX_n$ (given in detail below) whose sample size $n$ tends to infinity. The underlying filtration $(\cK_{t,n}: t\in [0,T])$ is the \v Cech or Vietoris-Rips complex of a Poisson or binomial point cloud in the critical regime. We apply these results to a smooth bootstrap procedure proposed in Roycraft et al. \cite{roycraft2020bootstrapping} and derive rates of convergence for the bootstrap procedure of the EC.

The quantification of the rates of convergence relies on normal approximations for general non-additive functionals in stochastic geometry. Such approximations are based on Stein's method and reduce to variance bounding tasks. These results are obtained from the seminal works of Chatterjee \cite{chatterjee2008new} and Lachi{\` e}ze-Rey and Peccati \cite{lachieze2017new} who derive normal approximations for the Kantorovich-Wasserstein distance and the Kolmogorov distance.

The EC is an important functional in persistent homology, the current major branch of TDA. 
To the best of our knowledge the present contribution is the first to obtain rates of convergence for this topological invariant.
Moreover, the presented FCLT is the first that also covers the case for \v Cech filtration, completing the investigation of Thomas and Owada \cite{thomas2019functional} who observed the additional technical difficulties that arise when proving the tightness of the processes in \eqref{E:EmpProcess} for the \v Cech complex. In addition, we extend the FCLT to a binomial sampling scheme by a Poissonisation argument, this setting has previously been not considered.

As a consequence of the continuous mapping theorem, the FCLT opens the door to obtain the asymptotic distribution of a variety of important functionals applied to the EC such as smoothed integral statistics or the running supremum.

Modern TDA is founded in the groundbreaking contributions of Edelsbrunner et al. \cite{edelsbrunner2000topological}, Zomorodian and Carlsson 
\cite{zomorodian2005computing} and Carlsson \cite{carlsson2009topology}. The presented central limit theorems rely heavily on the approach of Penrose and Yukich \cite{penrose2001central}, who define and use the idea of stabilizing functionals of Lee (see \cite{lee1997central, lee1999central}). This approach has also been used to establish central limit theorems for Betti numbers and persistent Betti numbers: Yogeshwaran et al. \cite{yogeshwaran2017random} were the first to establish a central limit theorem for Betti numbers from a stationary Poisson process with unit intensity. Hiraoka et al. \cite{hiraoka2018limit} extended this result to persistent Betti numbers from a stationary Poisson process. Krebs and Polonik \cite{krebs2019betti} established the strong stabilizing property of persistent Betti numbers and extended the validity of the central limit theorem to the binomial point process with a non-constant density. Krebs and Hirsch \cite{krebs2020cylinder} studied functional central limit theorems for persistent Betti numbers on a cylindrical domain. Betti numbers of $B$-bounded features have been studied by Biscio et al. \cite{svane}.

So far, bootstrap procedures in TDA mainly were based on the assumption that an iid sample of persistence diagrams is available; see Chazal et al. \cite{chazal2013bootstrap}, Fasy et al. \cite{fasy2014confidence}, Berry et al. \cite{berry2020}. Large sample results of these bootstrap methods are then shown for the number of available persistence diagrams, $N$, tending to infinity. An alternative set-up is to consider an individual persistence diagram, based on a single data set of size $n$. In this setting, that so far has received much less attention, the constructions are based on a Lipschitz-type property of the persistence diagram (also called `stability'). However, this approach only leads to asymptotically conservative results; see \cite{fasy2014confidence}. For another application of the bootstrap, see Shin et al. \cite{shinconfidence}.
For the upper level set filtration of a kernel density estimator, Chazal et al. \cite{chazal2014robust} successfully construct bootstrap confidence sets directly, without invoking stability.

Following the ideas of Roycraft et al. \cite{roycraft2020bootstrapping}, we here also consider the second approach, based on an individual sample of iid date of size $n$. In contrast to the existing results in the literature, this approach results in asymptotically valid bootstrap procedures. More precisely, given a point cloud in $\R^d$ of iid data points which are distributed according to an unknown density $\kappa$, our smooth bootstrap procedure relies on replicate point cloud data drawn from an appropriate density estimate $\hat\kappa$. We quantify the Kantorovich-Wasserstein distance and the Kolmogorov distance between the bootstrapped EC and the true EC in terms of the sample size and the supremum norm $\|\kappa-\hat\kappa\|_\infty$. Depending on the density estimator used, we obtain rates of convergence for our smooth bootstrap procedure. In the case of the Kantorovich-Wasserstein distance, we also show a rate of convergence in the functional setting, where the EC defines a curve on the interval $[0,T]$.

A simulation study is provided, investigating the small-sample properties of the bootstrap applied to the Euler characteristic curve in the functional setting.

The rest of the paper is organized as follows. Section~\ref{Section_Definitions} provides relevant notation and definitions. We present all main results in Section~\ref{Section_Results}. Technical details are given in Section~\ref{Section_TechnicalResults} and Appendix~\ref{Appendix}.

\section{Notation and definitions}\label{Section_Definitions}
We write $\N$ for the natural numbers $\{1,2,\ldots\}$ and $\N_0$ for $\N\cup\{0\}$.
The cardinality of a (finite) set $A$ is $\# A$. Let $P$ be a point process on $\R^d$ and $A\subset \R^d$. We write $P(A) = \# \{x\in P: x\in A \}$ for the random measure of $A$ under $P$. For $z\in\R^d$, $Q(z)$ denotes the cube $z + [-1/2 , 1/2]^d$. For $m\in\N $, the set $\{1,\ldots,m\}$  is abbreviated by $[m]$.
Moreover, we write $a \preceq b$, for $a,b\in\R^d$, if $a$ is equal to $b$ or if $a$ precedes $b$ in the lexicographic ordering, and $e_d$ for the all-one vector $(1,\ldots,1) \in \R^d$.

Write $B_n = [-n^{1/d} /2 , n^{1/d}/2]$. Let $\cP, \cP'$ be independent homogeneous Poisson processes on $\R^d$ with unit intensity. For $z\in\Z^d$, we let $\cF_z$ denote the $\sigma$-field $\sigma\{ \cP \cap Q(y): y \preceq z \}$ generated by the Poisson points of $\cP$ inside the cubes $Q(y)$ for all $y\in\Z^d$ equal to or preceding $z$ lexicographically. 

Next, we shortly introduce the main mathematical objects from TDA, which are needed in this paper. We refer to Boissonnat et al. \cite{boissonnat2018geometric} for a more thorough introduction to the subject.

Given a finite set $P$, an \textit{(abstract) simplicial complex} $\cK$ is a collection of non-empty subsets of $P$ which satisfy (i) if $x\in P$, then $\{x\} \in \cK$ and (ii) if $\sigma \in \cK$ and $\tau\subseteq \sigma$, then $\tau\in\cK$. If $\sigma\in \cK$ and $\# \sigma = k+1$, with $k\in\N_0$, then the simplex $\sigma$ has dimension $k$, viz., $\dim \sigma = k$.

The EC of a (finite) simplicial complex $\cK$ is given by the alternating sum of its simplex counts $S_k(\cK) = \# \{\sigma \in\cK: \dim \sigma = k \}$, viz.,
\[
		\chi(\cK) = \sum_{k\in\N_0} (-1)^k S_k(\cK).
\]
In this work, we consider the \v Cech and Vietoris-Rips complex constructed from point clouds in $\R^d$. Given a finite subset $P$ of the Euclidean space $\R^d$, the {\v C}ech filtration $\cC(P)=(\cC_t(P):t\ge 0)$ and the Vietoris-Rips filtration $\cR(P)=(\cR_t(P):t\ge 0)$ are defined through the following simplicial complexes
\begin{align*} 
				\cC_t(P) &= \{  \sigma \subseteq P, \bigcap_{x\in \sigma} B(x,t)\neq \emptyset \}, \\
				\cR_t(P) &= \{ \sigma \subseteq P, \diam(\sigma)\le t \},
\end{align*}
where $B(x,t) = \{y\in\R^d: \|x-y\|\le t\}$ is a closed Euclidean ball, and $\diam(\cdot)$ is the diameter of a set.

We use a generic notation and write $\cK_t$ for either the {\v C}ech or the Vietoris-Rips complex with parameter $t$ obtained from a random point cloud in $\R^d$.
The \v Cech filtration characterizes simplices in terms of the radius of their circumsphere, the smallest closed ball containing the given simplex. The Vietoris-Rips filtration relies on the pairwise distances between points in the simplex. This property not only makes the \v Cech filtration analytically more complex than the Vietoris-Rips filtration but also more computationally intensive to work with in practice. The filtration time of a simplex $\sigma$, written as $r(\sigma)$, is the smallest filtration parameter $t$ such that $\sigma$ is included in $\cK_t$. $r(\sigma)$ corresponds to the radius of the circumsphere of $\sigma$ in the \v Cech case, and to the maximum pairwise distance between the points of $\sigma$ for the Vietoris-Rips filtration.

Given an ordering of the simplices in the simplicial complex $\cK$, we can separate each $S_k(\cK)$ in positive simplices $S_k^+(\cK)$ and negative simplices $S_k^-(\cK)$, so $S_k(\cK) = S_k^+(\cK) + S_k^-(\cK)$. A $k$-simplex is positive if it creates a $k$-dimensional feature. It is negative if it kills a $(k-1)$-dimensional feature. The difference $S_k^+(\cK) - S_{k+1}^-(\cK)$ is the $k$th Betti number $\beta_k(\cK)$ of the simplicial complex $\cK$, see also \cite{boissonnat2018geometric}. Clearly, $S_0 = S_0^+$. This yields the well-known identity
\begin{align}
\begin{split}\label{ECBetti}
		&\chi(\cK) = S_0^+(\cK) + \sum_{k=1}^\infty (-1)^k S_k^+ (\cK) - \sum_{k=1}^\infty (-1)^{k-1} S_k^- (\cK) \\
		&= (S_0^+ (\cK) - S_1^- (\cK)) + \sum_{k=1}^\infty (-1)^k (S_k^+ (\cK) - S_{k+1}^- (\cK)) = \sum_{k=0}^{\infty} (-1)^k\beta_k (\cK).
		\end{split}
\end{align}
For \v Cech complexes of a $d$-dimensional point cloud, the right-hand side of \eqref{ECBetti} reduces to a sum of finitely many terms. This is because in $d$-dimensional space all Betti numbers $\beta_k$, $k\ge d$, are identically zero for the \v Cech complex, see also \cite{yogeshwaran2017random}. Using $0 = \beta_k(\cK) = S_k^+(\cK) - S_{k+1}^- (\cK)$ for all $k\ge d$, we see that in this case $\chi(\cK) = \sum_{k=0}^{d-1} (-1)^k S_k(\cK) + (-1)^d S_d^{-}(\cK)$.

Let $\kappa$ be a density function on $[0,1]^d$. Moreover, let $\mX_n = \{X_1,\ldots, X_n \}$ be a binomial process of length $n$, where the components $X_i$ are independently distributed with density $\kappa$. Furthermore, let $\cP_n$ be a non-homogenous Poisson process with intensity function $n\kappa$.

The underlying point cloud is allowed to be either the Poisson point cloud $n^{1/d} \cP_n$ or the scaled binomial point cloud $n^{1/d} \mX_n$. Here $\cK_{t, n}$ equals either $\cK_t(n^{1/d} \cP_n)$ or $\cK_t(n^{1/d} \mX_n)$ for a filtration parameter $t\in [0,T]$, $T<\infty$. We study the empirical process of the EC $(\ol \chi_n(t))_{t\in [0,T]}$ from \eqref{E:EmpProcess} based on these filtrations and point clouds. Obviously, $\ol\chi_n(0) \equiv 0$ for all $n\in\N$. To indicate the dependence on the density, we write $\ol \chi_{\kappa, n}$ for the EC obtained from an underlying density function $\kappa$. Moreover, $\cN_{0,1}$ denotes the law of the standard normal distribution.

The Kantorovich-Wasserstein distance between two Borel probability measures $\mu_1$ and $\mu_2$ on a metric space $(M,d)$ is
\begin{align*}
		d^M_W(\mu_1,\mu_2) &\coloneqq \sup \Biggl \{  \Biggl | \int f \intd \mu_1 - \int f \intd \mu_2 \Biggl |: \\
		&\qquad\qquad f\colon M\to \R \text{ is Lipschitz with Lipschitz constant } L_f \le 1 \Biggl \}.
\end{align*}
In the following, we study the Kantorovich-Wasserstein distance on $\R$, denoted by $d_W$.  We abbreviate the space of c{\` a}dl{\` a}g functions by $\fD = D([0,T])$. We write $d_W^{\fD}$ for the Kantorovich-Wasserstein distance on the normed vector space $( \fD, \| \cdot \|_\infty)$, where $\| \cdot \|_\infty$ is the sup-norm. The Kolmogorov distance between two Borel probability measures $\mu_1,\mu_2$ on the real numbers is
\[
		d_K(\mu_1,\mu_2) \coloneqq \sup\{ |\mu_1( (-\infty,u] ) - \mu_2( (-\infty,u] ) | : u\in\R \}.
\]
If $Z_1 \sim \mu_1$ and $Z_2 \sim \mu_2$, we also write $d_W(Z_1,\mu_2)$ and $d_W(Z_1,Z_2)$ for $d_W(\mu_1,\mu_2),$  and similarly for $d_K$. We adopt the same convention for the metric $d^{\fD}_W$ if $Z_1,Z_2$ are $\fD$-valued.

Let $\kappa$ be a bounded density function on $[0,1]^d$, i.e., $\|\kappa\|_\infty<\infty$. Write $B_{\infty}(\kappa,\rho)$ for the class of all densities $\nu$ on $[0,1]^d$ which satisfy $\|\kappa-\nu\|_\infty \le \rho$.

\section{Main results}\label{Section_Results}

\subsection{Approximation and central limit theorems}
We make use of the pioneering contributions of Chatterjee \cite{chatterjee2008new} (for the Kantorovich-Wasserstein distance) and of Lachi{\`e}ze-Rey and Peccati \cite{lachieze2017new} (for the Kolmogorov distance). We begin with an approximation result, which shows that the EC is locally Lipschitz-continuous in the underlying density function.
\begin{theorem}[Approximating property in the Kantorovich-Wasserstein distance]\label{T:ApproximationEuler}
Let $\kappa$ be a density on $[0,1]^d$ that is bounded, let $\rho\in\R_+$, and let $\nu\in B_{\infty}(\kappa,\rho)$. There are coupled Poisson processes $(\cP_n, \cQ_n )$ with intensities $(n\kappa, n\nu)$ and coupled binomial processes $(\mX_n,\mY_n)$ of length $n$ with densities $(\kappa,\nu)$, respectively, and a constant $C_{0,\kappa} \in \R_+$ depending on $\kappa$ and $\rho$ but not on $\nu$ (as long as $\nu\in B_\infty(\kappa,\rho)$), such that for all $n\in\N$
\begin{align}\label{E:ApproximationEuler0}
		\sup_{t\in [0,T] } \V( \ol \chi_{\kappa,n}(t) - \ol \chi_{\nu,n}(t) ) \le C_{0,\kappa } \| \kappa - \nu \|_{\infty}.
\end{align}
In particular, \[\sup_{t\in[0,T]} d_W (  \ol \chi_{\kappa,n}(t) ,  \ol \chi_{\nu,n}(t) ) \le \sqrt{C_{0,\kappa }} \| \kappa - \nu \|_{\infty}^{1/2}.\]
\end{theorem}
Note that the choice of $\rho$ in the above result is arbitrary and in particular independent of the density $\kappa$. It is known that the EC tends to a Gaussian process for a Poisson sampling scheme and the Vietoris-Rips filtration, see \cite{thomas2019functional}. We generalize this statement in the following to the \v Cech filtration and the binomial sampling scheme and quantify the convergence. The given rate is asymptotically optimal when compared to the classical result of Berry-Esseen for the normalized empirical mean of iid data which is of order $n^{-1/2}$. The main reasons for this fast rate are the stabilizing properties of the EC, see Proposition~\ref{P:StrongStabilizingProperty}, which correspond to $m$-dependent (and thus nearly iid) observations.
\begin{theorem}[Normal approximation]\label{T:NormalApproximation}
Consider the {\v C}ech or the Vietoris-Rips filtration as well as the Poisson or binomial sampling scheme.
Let $\kappa$ be a bounded density on $[0,1]^d$ and let $t\in (0,T]$. There is a $\rho>0$ and a corresponding $C_{1,\kappa} \in\R_+$, also depending on $t$, such that for all $\nu\in B_\infty(\kappa,\rho)$
\begin{align}\begin{split}\label{E:NormalApproximation}
		&\Biggl \{ d_K\left( \frac{ \ol \chi_{\nu,n}(t) }{ \V ( \ol \chi_{\nu,n}(t)) ^{1/2} } , \cN_{0,1} \right) \Biggl \} \\
		&\vee \Biggl \{  d_W\left( \frac{  \ol \chi_{\nu,n}(t) }{ \V ( \ol \chi_{\nu,n}(t)) ^{1/2} } ,\cN_{0,1} \right) \Biggl \} \le \frac{C_{1,\kappa} }{n^{1/2} }.
		\end{split}
\end{align}
Moreover, there are $C_{2,\kappa}, \wt C_{2,\kappa} \in\R_+$, also depending on $\rho, t$, such that
\begin{align}\label{E:KolmogorovApproximationEuler0}
	 d_K \left(  \ol \chi_{\kappa,n}(t) ,  \ol \chi_{\nu,n}(t) \right) \le \frac{C_{2,\kappa} }{n^{1/2} } +  \wt C_{2,\kappa}  \| \kappa - \nu \|_{\infty}^{1/2} .\end{align}
\end{theorem}
We detail the limiting covariance structure of the process $(\ol\chi_n(t))_{t\in [0,T]}$ in Theorem~\ref{T:FunctionalCLT} below. In particular, we show that $\lim_{n\to\infty} \V ( \ol \chi_{\kappa,n}(t)) \in\R_+$ for each $t>0$. In order to obtain \eqref{E:KolmogorovApproximationEuler0}, we need to quantify the quotient $\V ( \ol\chi_{\kappa,n}(t) )/ \V ( \ol\chi_{\nu,n}(t))$ and its inverse. Both quotients are meaningful if the limiting variance is bounded away from zero and infinity. 
Note that the normal approximation given above does not immediately extend to a uniform result.

To obtain a FCLT and rates of convergence that consider the entire EC on an interval $[0,T]$, we need an understanding of the continuity properties of the filtration time as a function of the underlying simplex. These depend on the simplicial complex in use  and we highlight this by writing $r_{\cC}(\cdot)$, resp., $r_{\cR}(\cdot)$, for the filtration time of the \v Cech, resp., Vietoris-Rips complex. We also write $r(\cdot)$ to refer to either of them if no distinction is necessary. Assume that $Z_0,Z_1,\ldots, Z_q$ are iid according to a density $\kappa$ on $[0,1]^d$ and let $\{Z_0, Z_1,\ldots,Z_q\}$ denote the $q$-simplex spanned by $Z_0,\ldots, Z_q$. If we use the Vietoris-Rips filtration, we can easily derive
\[
\p( r_{\cR}(\{ Z_0,Z_1,\ldots,Z_q \}) \in (a,b] ) \le \alpha_d \| \kappa \|_\infty  \ q(q+1) \ ( b^d - a^d),
\]
where $\alpha_d$ is the $d$-dimensional Lebesgue measure of the $d$-dimensional ball $B_d(0,1)$, see Lemma~\ref{L:ContinuityVR}.

If we instead use the {\v C}ech filtration, the situation is much more complex because it is no longer sufficient to study pairwise distances only. Instead, the filtration time is influenced by the geometry of the embedding space, $\R^d$, and is determined by the radius of the circumsphere. This radius can be calculated analytically with the result from Coxeter \cite{coxeter1930circumradius} using the Cayley-Menger matrix; we also refer to Le Ca{\"e}r \cite{le2017circumspheres} for more results on the circumsphere of $q+1$ points in $d$-dimensional Euclidean space. We obtain a similar result in Lemma~\ref{L:ContinuityCech},
\begin{align*}
		&\p( r_{\cC}(\{ Z_0,Z_1,\ldots,Z_q \}) \in (a,b] ) \\
		& \le (q+1)^{d+2}\ \max_{1\le m \le (d\wedge q)+1} ( \|\kappa\|_\infty \ \alpha_{d} \ d^{d/2} ) ^{m} \ \cdot \ \int_a^b g^*_d(t) \intd{t},
\end{align*}
for a certain continuous real-valued function $g^*_d$ depending on $d$ only.

With these preparations, we are now able to give the approximation property in the functional Kantorovich-Wasserstein distance $d_W^{\fD}$. Here it is of course necessary that $\sup_{t\in [0,T] } | \ol \chi_{\kappa,n}(t) - \ol \chi_{\nu,n}(t) |$ be measurable, which is true because the EC functional $t\mapsto \chi( \cK_{t,n} )$ is c{\` a}dl{\` a}g.

\begin{theorem}[Functional approximation]\label{T:FunctionalWasserstein}
Let $\kappa$ be a bounded density on $[0,1]^d$ and let $\rho\in\R_+$. Let $\nu\in B_\infty(\kappa,\rho)$. Consider the {\v C}ech or the Vietoris-Rips filtration. Let $[0,T]$ be partitioned into $J$ equidistant intervals of length $T/J$. 

 There are coupled Poisson processes $(\cP_n, \cQ_n)$ with intensities $(n\kappa, n\nu)$, coupled binomial processes $(\mX_n,\mY_n)$ of length $n$ with densities $(\kappa,\nu)$, respectively, and there are constants $C_{3,\kappa}, C_{4,\kappa}  \in \R_+$ depending on $\kappa, T>0$ and $\rho$ but neither on $\nu \in B_\infty(\kappa,\rho)$ , nor on $n$ nor on $J$, such that the following holds:
\begin{align}
\begin{split}\label{E:FunctionalApproximationEuler0}
		 &\E{ \sup_{t\in [0,T] } \Big| \ol \chi_{\kappa,n}(t) - \ol \chi_{\nu,n}(t) \Big|^2 } \\
		 & \le C_{3,\kappa } (J + T) \| \kappa - \nu \|_{\infty} + C_{4,\kappa } T^2 \ n \ J^{-1} \ \| \kappa - \nu \|_{\infty}^2.
		 \end{split}
\end{align}
In particular, there are constants $C_{5,\kappa}, C_{6,\kappa} \in \R_+$ depending on $\kappa, T>0$ and $\rho$, but neither on $\nu \in B_\infty(\kappa,\rho)$, nor on $n$ nor on $J,$ such that
\begin{align}
\begin{split}\label{E:FunctionalWassersteinApproximationEuler}
		&d^{\fD}_W (  (\ol \chi_{\kappa,n}(t))_{t\in [0,T]} ,  (\ol \chi_{\nu,n}(t))_{t\in [0,T]}  ) \\
		&\le C_{5,\kappa } \ J^{1/2} \ \| \kappa - \nu \|_{\infty}^{1/2}  +  C_{6,\kappa } \sqrt{ \frac{n}{J} } \| \kappa - \nu \|_{\infty}.
		\end{split}
\end{align}
\end{theorem}

Obviously, the result in \eqref{E:FunctionalWassersteinApproximationEuler} is also valid for general (uncoupled) Poisson processes $\cP_n,\cQ_n$ with intensity functions $n\kappa$, $n\nu$, and general $n$-binomial processes $\mX_n,\mY_n$ with density functions $\kappa,\nu$, respectively.

Moreover, using the continuity properties of the \v Cech filtration, we now extend the findings of \cite{thomas2019functional} who provide a functional central limit theorem for the Vietoris-Rips complex and a Poisson sampling scheme. We remark that a functional central limit theorem for the binomial sampling scheme has not been established yet for either filtration type and follows from a Poissonization argument covered in the technical details of Section~\ref{Section_TechnicalResults}. 

Using the strong stabilizing property of the EC from Proposition~\ref{P:StrongStabilizingProperty}, the following limits exist for each $t\in\R_+$ and $z\in\Z^d$ and can be expressed in terms of a finite and deterministic radius of stabilization
\begin{align}
	\begin{split}\label{D:DeltaInf}
	\Delta_\infty(t) &\coloneqq \lim_{n\to\infty} \chi( \cK_t( ( \cP \cup \{0\}) \cap B_n ) ) - \chi( \cK_t(  \cP \cap B_n )  ) \\
	&\ = \chi( \cK_t( ( \cP \cup \{0\}) \cap B(0,2t) ) ) - \chi( \cK_t(  \cP \cap B(0,2t) )  )  \quad a.s. 
	\end{split} \\
	 \fD_\infty(t,z) &\coloneqq \lim_{n\to\infty} \chi( \cK_t( \cP \cap B_n )) \nonumber \\
	 &\qquad\qquad - \chi( \cK_t(  [ (\cP \cap B_n) \setminus Q(z ) ] \cup [ \cP' \cap B_n \cap Q(z ) ] )  ) \nonumber \\
	 &=  \chi( \cK_t( \cP \cap B(z,2t+\sqrt{d}) ))  - \chi( \cK_t(  [ (\cP \cap B(z,2t+\sqrt{d}) ) \setminus Q(z ) ] \nonumber \\
	 &\qquad\qquad \cup [ \cP' \cap B(z,2t+\sqrt{d}) \cap Q(z ) ] )  ) \quad a.s. \label{D:DeltaInf2}
\end{align}

We assume the following technical condition for the FCLT. We call a density function {\em blocked} if it has the form $\sum_{i=1}^{m^d} b_i \mathds{1}_{A_i}$, where $m\in\N$, $b_1,\ldots,b_{m^d}\in\R_+$ and the $A_i$ are rectangular sets of the form $\times_{i=1}^d I_{i,j_i}$, where the $(I_{i,j})_{j=1}^m$ partition $[0,1]$ in intervals of length $m^{-1}$. The density function $\kappa$ is bounded on $[0,1]^d$ and we assume that there is a sequence of blocked density functions $(\kappa_n)_{n\in\N}$ each defined on $[0,1]^d$ with the property
\begin{align}\label{E:PropertyKappa}
		\lim_{n\to\infty} \| \kappa_n - \kappa \|_\infty = 0.
\end{align}
For example, if $\kappa$ can be approximated uniformly by continuous density functions, then it can also be approximated uniformly by blocked density functions.

We present the FCLT, which enables us to capture the dynamic topological evolution of Vietoris-Rips and {\v C}ech complex as the filtration time runs through a given interval $[0,T]$.
\begin{theorem}[Functional central limit theorem]\label{T:FunctionalCLT}
Let $T\in \R_+$.
Let the filtration be obtained either from the Vietoris-Rips or the \v Cech complex. Let $\kappa$ satisfy \eqref{E:PropertyKappa}. There is a Gaussian process $\fG=(\fG(t): t\in [0,T] )$ such that, as $n\to \infty$,
\begin{align*}
		( \ol\chi_n(t) : t\in [0,T] ) \to \fG \text{ in distribution in the Skorohod $J_1$-topology.}  
\end{align*}
The covariance structure of $\fG$ depends on the sampling scheme. In the Poisson sampling scheme,
\[
		\E{ \fG(s) \fG(t) } = \E{ \gamma( \kappa(Z)^{1/d} (s,t) )		},
\]
where the random variable $Z$ has density $\kappa$ and 
\begin{align}\label{E:DefinitionGamma}
	\gamma(s,t) =  \E{ \E{ \fD_\infty(s,0) \ | \ \cF_0} \ \E{\fD_\infty(t,0) \ | \ \cF_0}	}  
\end{align}
 for $s,t\in [0,\infty)$ and for $\cF_0$ being the $\sigma$-field generated by $\{ \cP \cap Q(z): z\preceq 0 \}$. Then $\sup_{0\le s,t \le T} \gamma(s,t)<\infty$ by the representation in \eqref{D:DeltaInf2}.
 
 In the binomial sampling scheme
\[
		\E{ \fG(s) \fG(t) } =  \E{ \gamma\big( \kappa(Z)^{1/d} (s,t) \big)		} - \E{\alpha(\kappa(Z)^{1/d}s)}\E{\alpha(\kappa(Z)^{1/d}t)},
\]
where $\alpha(t) = \mathbb{E}[ \Delta_\infty ( t ) ] $. Furthermore, $\E{\fG(t)^2} > 0$ for all $ t\in (0,T]$ for both sampling schemes.

For both the Poisson and the binomial sampling scheme, the process $\fG$ has a continuous modification which is $\beta$-H{\"o}lder continuous for each $\beta\in (0,1/2)$.
\end{theorem}

An immediate consequence of the functional central limit theorem is the weak convergence of continuous functionals applied to the EC curve. Let $(S,d_S)$ be a metric space and let $\fJ \colon D([0,T]) \to (S,d_S)$ be continuous (w.r.t.\ to $d_S$ and the $J_1$-topology). Then, under the assumptions of the above theorem, $( \fJ( \ol\chi_n(t)): t\in [0,T] )$ converges weakly to $\fJ(\fG)$ as $n\to\infty$.

As an example, consider the smooth EC-transform, which is the image of $\ol \chi_n$ under the continuous integration mapping
\[
	\fI \colon D([0,T]) \to (C([0,T]), \|\cdot \|_\infty),\quad f\mapsto \int_0^\bullet f(s) \ \intd{s}. 
\]
Crawford et al. \cite{crawford2016functional} consider a similar transform of the EC curve with practical applications in functional data analysis. Further potential applications of the smooth EC-transform $\fI(\ol\chi_n)$ are goodness-of-fit tests as an exploratory tool in topological data analysis. We refer to \cite{svane} and \cite{krebs2020cylinder} for similar applications in the context of persistent Betti numbers.

\subsection{The bootstrap}
\label{section::bootstrap}

Our bootstrap procedure merely requires an estimate for the true density function $\kappa$ of the random variables $X_i$ underlying the Poisson or binomial process. Denote this estimate by $\hat\kappa_n$, where the index $n$ refers to the sample $\cP_n$, resp.\ $\mX_n$. So when considering $\cP_n$, we assume implicitely the knowledge of the Poisson parameter of $N_n$, which is $n$. For instance, $\hat\kappa_n$ can be obtained from a kernel density estimate, see \cite{mack1982weak} and \cite{hansen2008uniform}.

The bootstrap procedure works as follows: Conditional on the sample $\cP_n$ or $\mX_n$ and the density estimate $\hat\kappa_n$, we resample a Poisson process $\cP^*_n = \{ X^*_1,\ldots, X^*_{N^*_n} \}$ or a binomial process $\mX^*_n = \{ X^*_1 , \ldots, X^*_n \} $, where the $X^*_i$ are iid with density $\hat\kappa_n$ and the random variable $N^*_n$ is independent (of all other random variables) and Poisson distributed with mean $n$.

Using the sample $\cP^*_n$ or $\mX^*_n$, we compute the EC of the corresponding {\v C}ech or Vietoris-Rips complex $\cK^*_t$, which is either equal to $\cK_t(n^{1/d}  \cP^*_n)$ or to $\cK_t(n^{1/d} \mX^*_n)$, $t\in [0,T]$. The related empirical process is
\begin{align}
		\ol \chi^*_n(t) = n^{-1/2} \big( \chi( \cK^*_t ) - \Ec{ \chi( \cK^*_t )} \big), \quad t\in [0,T],
\end{align}
where $\mathbb{E}^*$ denotes the expectation conditional on the sample $\cP_n$ or $\mX_n$, respectively.
In practice we use a kernel estimate $\hat\kappa_n$; this smooth bootstrap is proposed in \cite{roycraft2020bootstrapping}. In that contribution we also address in detail possible problems with the ``standard'' bootstrap from the empirical distribution, which we sketch in the following. Hence, the present approach is an alternative, even though estimation of the true underlying density $\kappa$ can be challenging, especially in high dimensions. 

When compared to the direct bootstrap from the empirical distribution, our smooth bootstrap procedure has certain advantages. As the empirical distribution is discrete, the number of unique values in a given bootstrap sample is random and strictly smaller than $n$, with an expected number of points approximately $0.632n$. This can be problematic because in the critical regime, we rescale according to sample size by a factor of $n^{1/d}$. 

Moreover, since the support of the empirical distribution is discrete, the developed asymptotic theory does not apply, requiring at least an underlying distribution with a density. As such, there is a need for a smooth bootstrap procedure; we refer to \cite{roycraft2020bootstrapping} for a more thorough discussion with examples. Our first result applies to the EC evaluated at a specific point $t$.

\begin{theorem}[Pointwise validity of the bootstrap]\label{T:PointwiseValidity}
Let $\kappa$ be a bounded density on $[0,1]^d$ and let $(\hat{\kappa}_n: n\in \N)$ be a sequence of density estimators of $\kappa$ with the property that $\lim_{n\to\infty} \|\kappa - \hat\kappa_n \|_\infty =0$ a.s. (in probability). Then
\[
	 \|\hat\kappa_n - \kappa \|_\infty^{-1/2} \ \cdot \  \sup_{t\in [0,T]} d_W ( \ol \chi^*_n (t),  \ol \chi_n (t) ) = O(1) \quad a.s. \text{ (in probability)}.
\]
Furthermore, for each $t\in [0,T]$
\[
	\big\{ \|\hat\kappa_n - \kappa \|_\infty^{1/2} +  n^{-1/2} \big\}^{-1}  \ \cdot \ d_K ( \ol \chi^*_n (t),  \ol \chi_n (t) ) = O(1)  \quad a.s. \text{ (in probability)}.
\]
\end{theorem}
Consider the case of $n$ iid data points $Z_i$, where the density $\kappa$ has a continuous $p$\textsuperscript{th} derivative on $[0,1]^d$ and where the kernel density estimate $\hat\kappa_n$ is obtained from a $p$\textsuperscript{th} order kernel for an integer $p\ge 1$ (see \cite{tsybakov2008introduction} for the definition of the order of a kernel). In this case,
\begin{align}\label{E:KernelRate}
		 \|\hat\kappa_n - \kappa \|_\infty = O\big( (	n^{-1} \log n )^{p/(d+2p)} \big) \quad a.s.,
\end{align}
see e.g. \cite{hansen2008uniform}. Hence, for the Kantorovich-Wasserstein distance
\[
			 \sup_{t\in [0,T]} d_W ( \ol \chi^*_n (t),  \ol \chi_n (t) ) = O\big( (	n^{-1} \log n )^{p/(2d+4p)} \big) \quad a.s. 
\]
A similar result is true for the Kolmogorov distance for fixed $t$ (and not uniformly in $t\in [0,T]$), viz.,
\[
		 d_K ( \ol \chi^*_n (t),  \ol \chi_n (t) )= O\big( (	n^{-1} \log n )^{p/(2d+4p)} \big) \quad a.s.
\]
for each $t\in [0,T]$. Moreover, we have the following functional result.
\begin{theorem}[Functional validity of the smooth bootstrap]\label{T:Validity}
Let the assumptions of Theorem~\ref{T:PointwiseValidity} be satisfied. Additionally, let $(J_n:n\in\N)$ diverge to infinity such that   
\[
    	J_n \|\hat\kappa_n - \kappa\|_\infty \to 0\;\; a.s.\; (in\; probability)
	\]
and
\[
	\sqrt{ \frac{n}{J_n} }   \| \hat\kappa_n - \kappa \|_{\infty} \to 0 \quad a.s. \ \text{ (in probability)}
\]
as $n\to\infty$. Set $b_n = J_n^{1/2} \ \| \hat\kappa_n - \kappa \|_{\infty}^{1/2}  +  \sqrt{ n \ J_n^{-1} }   \| \hat\kappa_n - \kappa \|_{\infty} $. Then 
\begin{align*}
		&  b_n^{-1} \ \cdot \ d^{\fD}_W (  (\ol \chi^*_{n}(t))_{t\in [0,T]}, (\ol \chi_{n}(t))_{t\in [0,T]}  ) = O(1) \quad  a.s. \text{ (in probability)}.
\end{align*}
\end{theorem}
If we use a kernel estimator $\hat\kappa_n$ for $\kappa$, we obtain a consistent uniform bootstrap approximation given that the density is sufficiently smooth, in the above sense, with $p>d$. Set $J_n = (\log n)^{p/(2d+4p)} n^{(d+p)/(2d+4p)}$. Using \eqref{E:KernelRate}, 
\[
	d^{\fD}_W (  (\ol \chi^*_{n}(t))_{t\in [0,T]}, (\ol \chi_{n}(t))_{t\in [0,T]}  ) = O( (\log n)^{\alpha} n^{-\beta} ) \quad a.s.,
\]
where $\alpha = 3p/(4d+2p) >0$ and $1/8\ge \beta = 3p/(4d + 8p) -1/4>0$.

\subsection{Simulation study} \label{section::simulation_study}

In this section, we provide the results for a series of simulations using the smoothed bootstrap procedure described in Section~\ref{section::bootstrap}, establishing its efficacy in producing valid uniform confidence bands for the mean Euler characteristic curve of the \v Cech complex. Due to computational constraints, data generating distributions were chosen in dimensions $2$ and $3$ only. A description of the distributions considered is given in Table~\ref{table::distributions}. Visual illustrations are given in Figure~\ref{figure::dist_illustrations} for $F_1$ to $F_4$ in dimension $2$.

\begin{figure}
	\begin{center}
		\includegraphics[width=.22\textwidth]{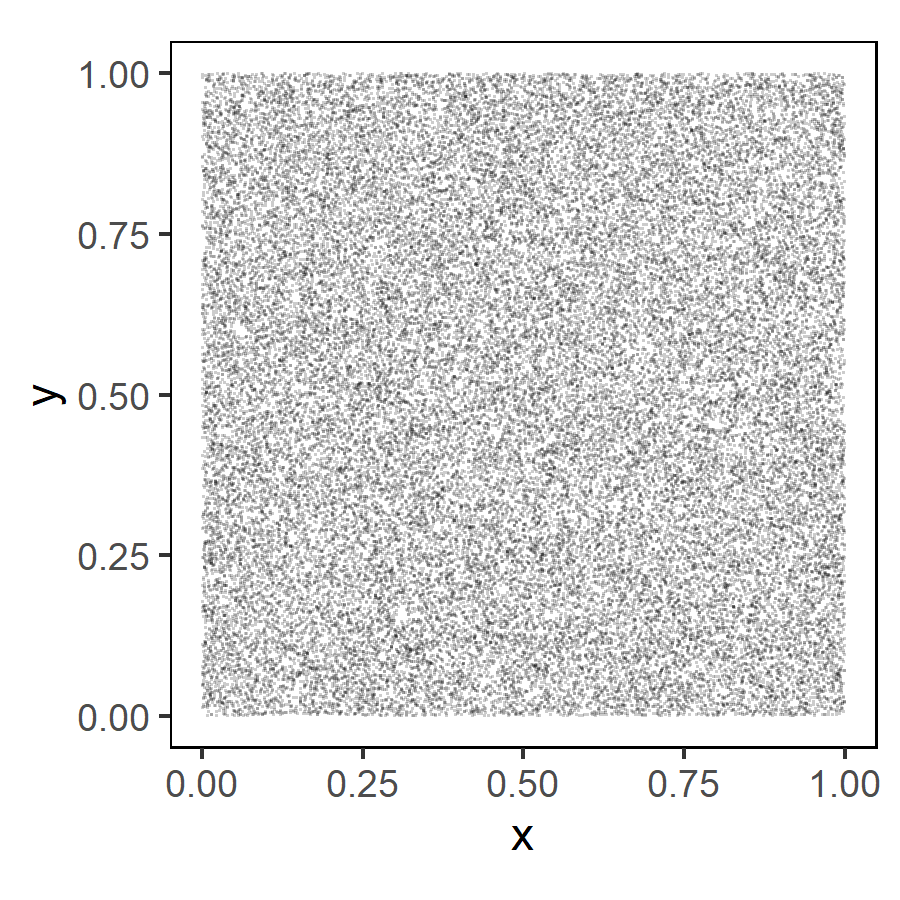}
		\includegraphics[width=.22\textwidth]{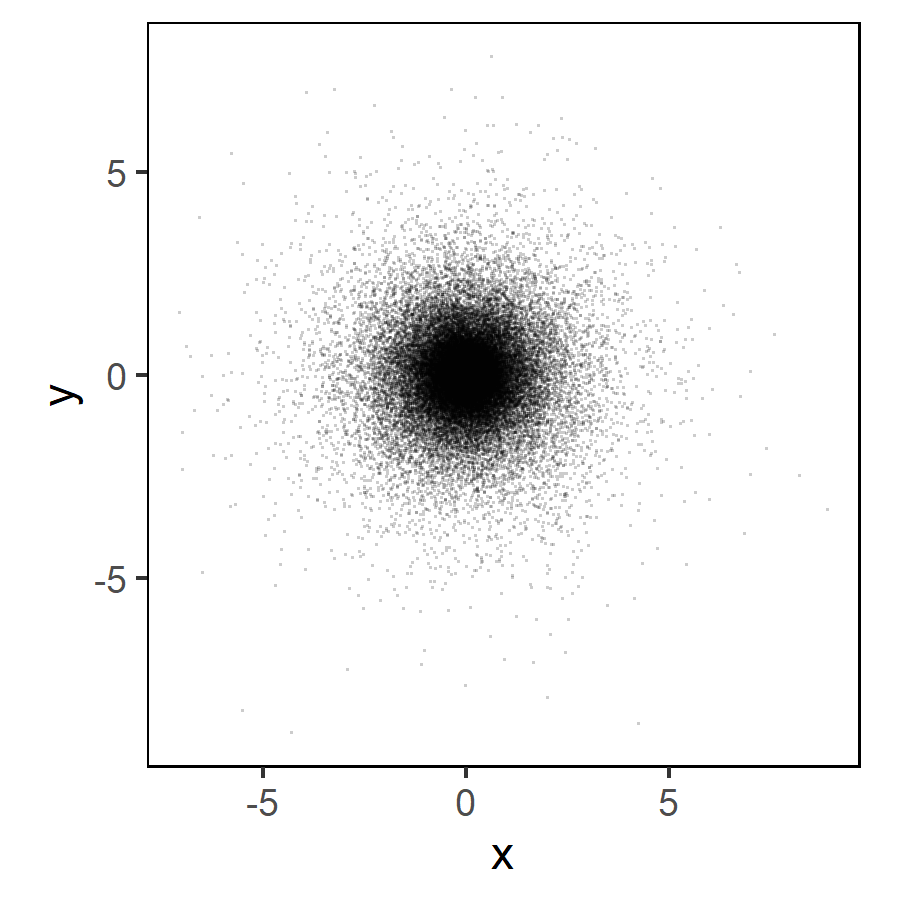}
		\includegraphics[width=.22\textwidth]{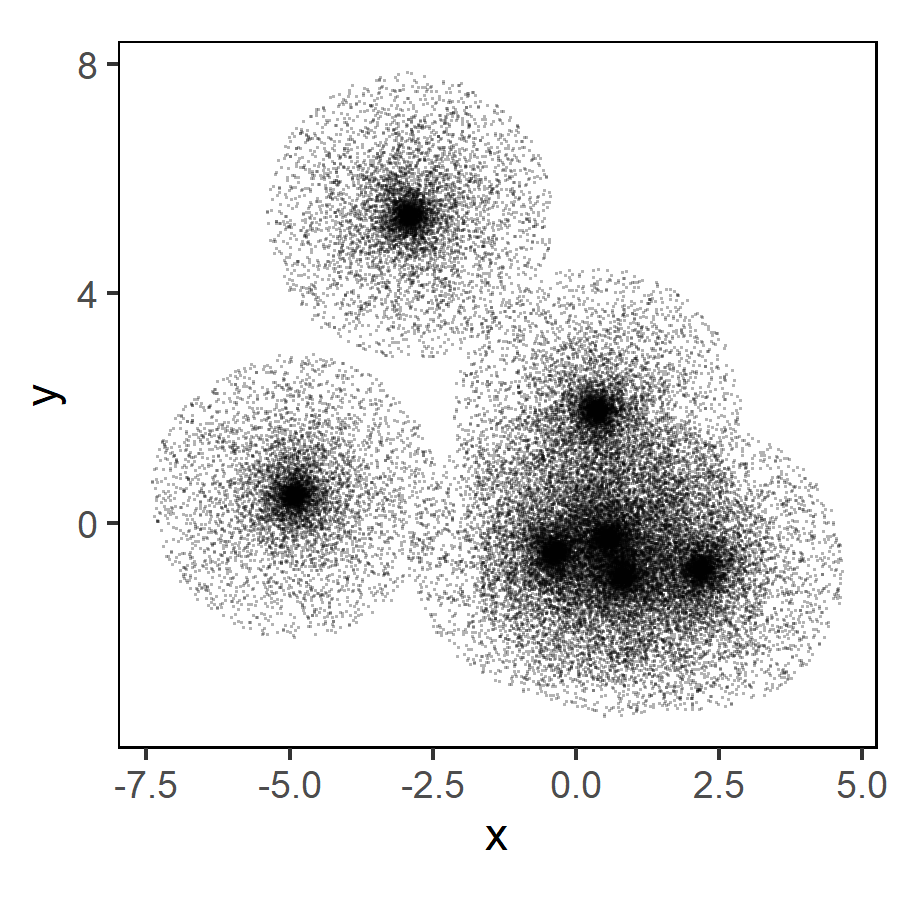}
		\includegraphics[width=.22\textwidth]{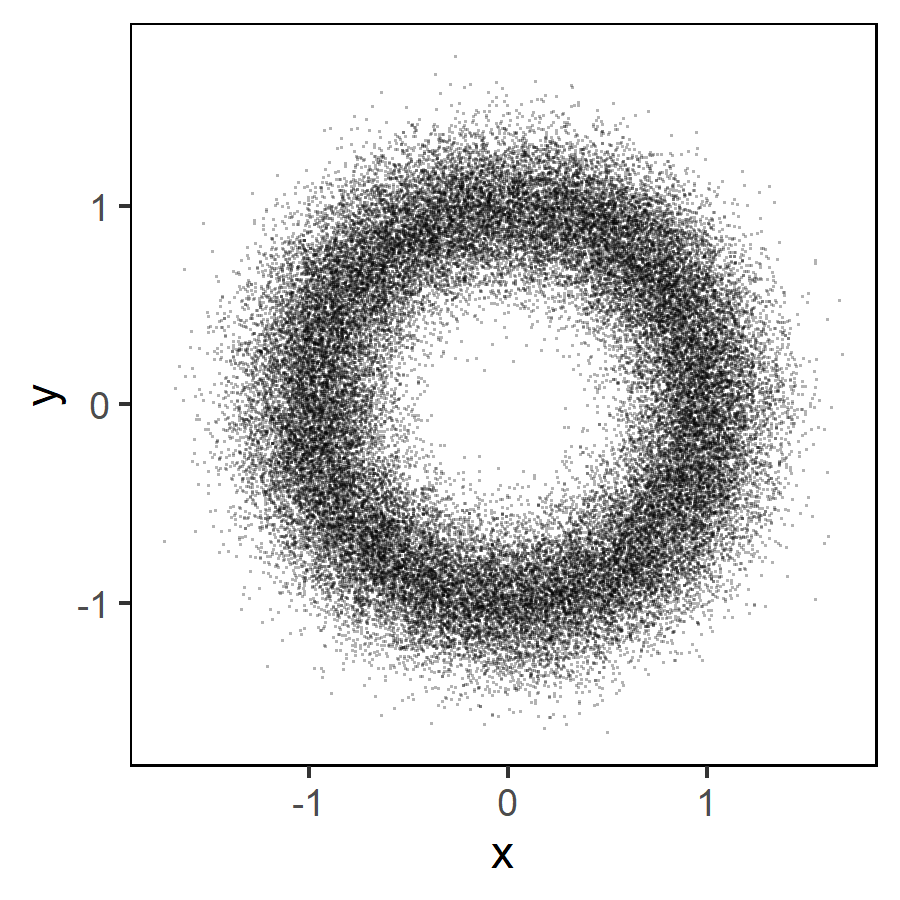}
	\end{center}
	\caption{Large-sample illustrations of distributions $F_1$-$F_4$ from Table~\ref{table::distributions}, shown from left to right, respectively. \label{figure::dist_illustrations}}
\end{figure}

\begin{table}
	\begin{center}
		\begin{tabular}{|c|c|l|}
			\hline
			Distribution & Dimension & Description \\
			\hline
			$F_1$ & 2 & Uniform on $[0, 1]^2$ \\
			$F_2$ & 2 & Rotationally symmetric, Exp$(1)$ radius \\
			$F_3$ & 2 & Density with 7 poles, bounded $L_2$ norm\\
			$F_4$ & 2 & Uniform on $\mathbb S^1$, additive Gaussian noise \\
			$F_5$ & 3 & Uniform on $[0, 1]^3$\\
			$F_6$ & 3 & \begin{minipage}[t]{.4\textwidth} Linked copies of $\mathbb S^1$, \\ uniform sampling with additive Gaussian noise.\vspace{.4em} \end{minipage} \\
			$F_7$ & 3 & \begin{minipage}[t]{.4\textwidth} $\mathbb S^2$ with smaller-radius $\mathbb S^1$ handle (kettlebell), \\ uniform sampling with additive Gaussian noise \end{minipage} \\
			\hline
		\end{tabular}
	\end{center}
	\caption{Distributions considered for the simulation study of Section~\ref{section::simulation_study} \label{table::distributions}}
\end{table}

For a given distribution and sample size, the true mean curve of the Euler characteristic $(\E{\chi(\mathcal K_t(n^{1/d}\mathbb X_n))} : t\in[0, T])$ was estimated using the average over a large ($n_{\mu}=50000$) number of iid replicates from the true distribution. Betti number calculations for the \v Cech complex were done using the GUDHI library via alphaComplexDiag from the TDA R package. Evaluation was done at a dense ($n_t=1000$) grid within $[0, T]$, with the exact value of $T$ changing depending on sample size and distribution. $T$ was chosen large enough as to not influence the analysis. The estimation error included in these steps is considered negligible.

Next, for a given sample size, we generate an original sample $\mathbb X_n$, and $B=1000$ bootstrap replications, using the smoothed bootstrap procedure. Bandwidth selection was done using Hpi.diag from the ks R package. 

The mean curve $(\Ec{\chi(\mathcal K_t(n^{1/d}\mathbb X^*_n))}: t\in[0, T])$ was estimated using the average Euler curve over the $B$ bootstrap replicates, again evaluated at a dense ($n_t=1000$) grid within $[0, T]$. For each bootstrap sample $\mathbb X^*_{n, i}$, we calculate
\[e_i=\sup_{t\in[0, T]}\left|\chi\left(\mathcal K_t(n^{1/d}\mathbb X^*_{n, i})\right)-\frac{1}{B}\sum_{i=1}^B\chi\left(\mathcal K_t(n^{1/d}\mathbb X^*_{n, i})\right)\right|.\]

To establish coverage, the $0.95$ quantile of the $e_1, ..., e_B$ gives the width of the corresponding uniform confidence band, and is compared to 
$$
	\sup_{t\in[0, T]}\left|\chi\left(\mathcal K_t(n^{1/d}\mathbb X_n)\right)-\E{\chi(\mathcal K_t(n^{1/d}\mathbb X_n))}\right|
$$
using the established estimate of the true mean curve. The entire data generation and bootstrap procedure was repeated $n_p=500$ times to estimate the coverage proportion. Coverage proportions and details are provided in Table~\ref{table::coverage}.

\begin{table}[]
	\centering
	\begin{tabular}{|ll|lllll|}
		\hline
		\multicolumn{2}{|l|}{\multirow{2}{*}{}} & \multicolumn{5}{c|}{Sample Size ($n$)} \\
		\multicolumn{2}{|l|}{}                  & $30$   & $50$  & $100$ & $200$ & $500$ \\ \cline{3-7} 
		\multicolumn{1}{|c}{}      & $F_1$      & 0.980  & 0.972 & 0.978 & 0.978 & 0.978 \\
		& $F_2$      & 0.948  & 0.954 & {\color{red} 0.904} & {\color{red} 0.884} & {\color{red} 0.880} \\
		& $F_3$      & 0.978  & 0.962 & 0.966 & 0.978 & 0.974 \\
		Distribution               & $F_4$      & 0.964  & 0.974 & 0.972 & 0.964 & 0.960 \\
		& $F_5$      & {\color{red} 0.914}  & {\color{red} 0.922} & 0.958 & 0.958 & 0.968 \\
		& $F_6$      & 0.980  & 0.982 & 0.990 & 0.968 & 0.956 \\
		& $F_7$      & 0.984  & 0.980 & 0.986 & 0.988 & 0.982 \\ \hline
	\end{tabular}
	\caption{Coverage proportions for the bootstrap simulation study. Red color indicates situations of interest}
	\label{table::coverage}
\end{table}  

We see that the bootstrap procedure is generally conservative, yielding higher than the nominal $95\%$ coverage proportion in the majority of cases, moving towards the nominal level with larger sample size. In the case of $F_5$, the uniform distribution on $[0,1]^3$, the poor coverage for $n=30$ and $n=50$ is likely due to boundary effects not present in the other continuous cases. The stand-out case, however, is $F_2$, which seems to diverge from the stated level for large samples. In this case, the density approaches $\infty$ towards the origin, in such a way that no $L_p$ norm is bounded. This is likely the driving factor behind the poor coverage in this case.

For the results in this work, we consider only the case of a bounded density on $[0,1]$. As shown by the provided coverage proportions, it is likely that these conditions can be greatly weakened, while still providing for bootstrap consistency.

\section{Technical results}\label{Section_TechnicalResults}
Throughout all our proofs, we will use the same terminology and notation. In the following lines, we introduce more definitions which are exclusively needed in this section and in the appendix.

\textit{Convention about the connectivity.} Since we are studying simplicial complexes built from the {\v C}ech and the Vietoris-Rips filtration for filtration parameters in the range $[0,T]$, an upper bound on the diameter of the simplex is $2T$, resp.\ $T$. We abbreviate this upper bound by $\delta$, e.g., we only need to know the points in a $\delta$-neighborhood of a given point $x$ in order to determine the simplices containing $x$.

\textit{Convention about the densities.} Throughout this section and the appendix $\kappa$ is an arbitrary but fixed bounded density on $[0,1]^d$. Moreover, for a given $\rho\in \R_+$, we study density functions $\nu\in B_\infty(\kappa,\rho)$. The choice of the neighborhood parameter $\rho$ can depend on $\kappa$, however, this will then be mentioned. As already pointed out in Section~\ref{Section_Results}, the constants depend then on $\kappa$ and $\rho$ but not on $\nu\in B_\infty(\kappa,\rho)$.

\textit{Convention about constants.} To ease notation, most constants in this paper will be denoted by $c,c',C$, etc. and their values may change from line to line. These constants may depend on parameters like the dimension and often we will not point out this dependence explicitly; however, none of these constants will depend on the index $n$, used to index infinite sequences, or on the index $i$, used to index martingale differences. Furthermore, these constants will not depend on $\nu$ as long as $\nu$ satisfies $\|\nu -\kappa \|_\infty \le \rho$. If we point out this property explicitly, we say "$C$ is independent* of $\nu$". Specific constants carry a subscript $C_1, c_1$ etc.

\textit{Notation in the Poisson sampling scheme.} Let $\cP, \cP'$ be independent Poisson processes on $\R^d\times [0,\infty)$ with unit intensity. We assume the following couplings
\begin{align*}
		&  \cP(n) = \left\{	x \in \R^d: \exists t \le \kappa(x/n^{1/d} + e_d /2 ), (x,t)\in \cP	\right\}, \\
		& \cQ(n) = \left\{	x \in \R^d: \exists t \le \nu(x/n^{1/d} + e_d /2 ), (x,t)\in \cP	\right\}, \\
		&  \cP'(n) = \left\{	x \in \R^d: \exists t \le \kappa(x/n^{1/d} + e_d /2 ), (x,t)\in \cP'	\right\}, \\ 			& \cQ'(n) = \left\{	x \in \R^d: \exists t \le \nu(x/n^{1/d} + e_d /2), (x,t)\in \cP'	\right\}.
\end{align*}
Note that as in Section~\ref{Section_Results} the density $\kappa$ is related to the Poisson processes $\cP(n), \cP'(n)$ whereas the density $\nu$ belongs to $\cQ(n), \cQ'(n)$. The Poisson processes $\cP(n), \cP'(n)$ and $\cQ(n), \cQ'(n)$ are supported on the cube $B_n = [-n^{1/d}/2, n^{1/d}/2 ]^d$ and have intensity functions $\kappa(\cdot/n^{1/d} + e_d /2 )$ and $\nu(\cdot/n^{1/d} + e_d /2 )$, respectively, where here $e_d$ is the all-one vector $(1,\ldots,1) \in \R^d$.

Recall that $\cP_n$ (resp. $\cQ_n$) is a non-homogenous Poisson process with intensity function $n\kappa$ (resp. $n\nu$). Obviously, the distribution of $n^{1/d} \cP_n$ and $\cP(n)$ are equal modulo the shift; the same holds for $n^{1/d} \cQ_n$ and $\cQ(n)$. We write $B'_n = \{ z \in \Z^d: Q(z) \cap B_n \neq \emptyset\}$ and denote the  cardinality of $B'_n$ by ${b_n^\prime}$. We will use an enumeration of $B'_n$ given by $\{ z_{n,i}: i\in [{b_n^\prime}] \}$, where $z_{n,i} \preceq z_{n,i+1}$. Clearly, ${b_n^\prime} / n \to 1$ as $n\to\infty$.

For $A \subseteq [n]$, we write
$$\cP'^A(n) \coloneqq \big\{  \cP(n) \setminus \big( \bigcup_{i\in A} Q(z_{n,i}) \big) \big\} \cup \big\{ \cP'(n) \cap \big( \bigcup_{i \in A} Q(z_{n,i}) \big) \big\}.$$
In slight abuse of notation we write $\cP'^{\,i}(n)$ rather than $\cP'^{\{i\}}(n)$, and we write $\cP'^{\,z}(n)$ when replacing the points in $\cP(n) \cap Q(z)$ by points in $\cP'(n).$ We also use a similar notation with $\cP$ replaced by $\cQ$.

The following filtrations of simplicial complexes will be used to construct martingale differences:
\begin{alignat*}{4}
 &\cK_{t,n}  &&\coloneqq \cK_t( \cP(n) ) , \qquad
		 && \wt \cK_{t,n}  &&\coloneqq \cK_t( \cQ(n) ) ,  \\
		 &\cK'_{t,n,i}  &&\coloneqq \cK_t( \cP'^{\,i}(n) ) , \qquad
		&& \wt \cK'_{t,n,i}  &&\coloneqq \cK_t( \cQ'^{\,i}(n) ),
\end{alignat*}		
for $t\in [0,T]$, $i\in [{b_n^\prime}]$ and $n\in \N$. The next two filtrations are needed for approximation arguments:	 
\begin{align*}
		 & \wt \cK^*_{t,n,i}  \coloneqq \cK_t( [\cQ(n)\setminus Q(z_{n,i})] \cup [\cP(n) \cap Q(z_{n,i})]  ), \\
		   & \wt \cK'^*_{t,n,i} \coloneqq \cK_t( [\cQ(n)\setminus Q(z_{n,i})] \cup [\cP'(n) \cap Q(z_{n,i})]  ),
\end{align*}
for $t\in [0,T]$, $i\in [{b_n^\prime}]$ and $n\in \N$. Notice that filtrations without a ``tilde'' in their notation are based on $\cP(n)$ (and $\cP'(n))$, while those with a ``tilde'' are based on $\cQ(n)$ (and $\cQ'(n)$). The notation ``tilde-star'' indicates filtrations based on $\cQ(n)$ with small parts replaced by points in $\cP(n)$ and  $\cP'(n)$, respectively.

For each $n$ define a filtration of $\sigma$-fields by 
\begin{align}\label{DefSigmaField}
\cG_{n,j} = \sigma \{ \cP(n) \cap Q(z_{n,k}), \cQ(n) \cap Q(z_{n,k}):  z_{n,k} \preceq z_{n,j} \},
\end{align}
for $ j \in [ {b_n^\prime} ]$. Also set  $\cG_{n,0} = \{ \emptyset, \Omega\}$. The following notation is convenient for statements regarding the asymptotic normality. We write
\begin{align*}
			\fD'_n(t,z) = \chi(\cK_t( \cP(n) )) - \chi(\cK_t( \cP'^{\,z}(n) ) ), \\
			\wt \fD'_n(t,z) = \chi(\cK_t( \cQ(n) )) - \chi(\cK_t( \cQ'^{\,z}(n) ) ), 
\end{align*}
for first order differences in a specific point $z\in\Z^d$. Moreover, we use the following notation for first order differences tied to specific indices for $A\subseteq [{b_n^\prime}]$ and $j\in [{b_n^\prime}]$:
\begin{align*}
	\fD'^A_n( t,j ) &\coloneqq \chi( \cK_t( \cP'^A(n)) ) - \chi( \cK_t (\cP'^{A \cup \{j\} }(n) ) ),\\
	\wt \fD'^A_n( t,j ) &\coloneqq \chi( \cK_t( \cQ'^A(n)) ) - \chi( \cK_t (\cQ'^{A \cup \{j\} }(n) ) ).
\end{align*}
If $A=\emptyset$, we omit $A$ in the superscript on the left-hand side, so that $\fD'_n( t,j ) = \fD'_n( t,z_{n,j} ) = \chi(\cK_{t,n}) - \chi(\cK'_{t,n,j}).$

\textit{Notation in the binomial sampling scheme.} In order to ease the notation, we set ${b_n^\prime} :\equiv n$, so that we can treat the binomial and the Poisson sampling scheme with the same notation. We use coupled binomial processes $\mX = (X_i: i\in \N) ,\mX' = (X'_i: i\in \N),\mY = (Y_i: i\in \N),\mY' = (Y'_i: i\in \N )$ instead. These have the property that $(\mX,\mY)$ and $(\mX',\mY')$ are independent and the components of $\mX,\mX'$ and $\mY,\mY'$ have a density $\kappa$ and $\nu$, respectively, such that
\[
		\p( X_i \neq Y_i )  \vee \p( X'_i \neq Y'_i ) =  \frac{1}{2} \| \kappa - \nu \|_{TV} = \frac{1}{4} \|\kappa - \nu \|_1 \le \frac{1}{4} \|\kappa - \nu \|_{\infty} , \qquad \forall i \in \N,
\]
see for instance \cite{den2012probability}, Theorem 2.12. (Later, we will apply this coupling to the case where $\|\kappa-\nu\|_{\infty}$ is small.) 

In what follows, we will use the fact that the binomial processes are defined as sequences and not as point clouds. Define the filtration of $\sigma$-fields $\cG_{n,i} = \sigma\{ X_j, Y_j: j\in [i] \}$ for $i\in [n]$ and $\cG_{n,0} = \{\emptyset, \Omega\}$. Also, write $\mX_n$ for the elements $X_i$ of $\mX$ with $i\in [n]$, and similarly define $\mY_n, \mX'_n$ and $\mY'_n$. Furthermore, let
\[
	\mX'^{A}_{n} \coloneqq  ( \mX_n \setminus \{X_i , i \in A\} ) \cup \{ X'_i , i \in A\}, \quad A\subseteq [n].
\]
We write $\mX'^{\,i}_{n}$ for $\mX'^{\{i\}}_{n}$. A similar notation is used with $\mX_n$ replaced by $\mY_n.$ The following definitions of filtrations of simplicial complexes parallel those of the Poisson case:
\begin{alignat*}{4}
		& \cK_{t,n}  &&\coloneqq \cK_t( n^{1/d} \mX_n ) , \qquad\quad
		&& \wt \cK_{t,n}  &&\coloneqq \cK_t( n^{1/d} \mY_n ) ,  \\
		& \cK'_{t,n,i}  &&\coloneqq \cK_t(n^{1/d} \ \mX'^{\,i}_n,) , 
		&& \wt \cK'_{t,n,i}  &&\coloneqq \cK_t(n^{1/d} (  \mY'^{\,i}_n ) ),
\end{alignat*}
for $t\in [0,T]$, $i\in [n] $ and $n\in\N$, as well as
\begin{align*}
		 & \wt \cK^*_{t,n,i}  \coloneqq \cK_t( n^{1/d} ( [\mY_n \setminus \{Y_i\}] \cup\{X_i\}  ) ), 
		  &  \wt \cK'^*_{t,n,i}  \coloneqq \cK_t( n^{1/d} ([\mY_n \setminus \{Y_i\} ] \cup\{X'_i\}  ) ),
\end{align*}
for $t\in [0,T]$, $i\in [n] $ and $n\in\N$. Compared to the Poisson case, we replace $\cP(n), \cP'(n)$ by $n^{1/d} \mX_n, n^{1/d}\mX'_n$ and $\cQ_n, \cQ'_n$ by $n^{1/d}\mY_n,n^{1/d}\mY'_n$, respectively.  For $A\subseteq [{b_n^\prime}] = [n]$ and $j\in [{b_n^\prime}]$, we set 
\begin{align*}
\fD'^A_n( t,j) =\chi( \cK_t( \mX'^A_n ) ) - \chi( \cK_t ( \mX'^{A\cup \{j\}}_n ) ),\\
\wt\fD'^A_n( t,j) =\chi( \cK_t( \mY'^A_n ) ) - \chi( \cK_t ( \mY'^{A\cup \{j\}}_n ) ).
\end{align*}
Again, if $A=\emptyset$, we omit $A$ in the superscript on the left-hand side. 

Recall that in the Poisson sampling scheme the processes $\ol\chi_{n}$ were defined in Section~\ref{Section_Definitions} from the Poisson processes $n^{1/d}\cP_n$ (resp.\ $n^{1/d}\cQ_n$) and not $\cP(n)$ (resp.\ $\cQ(n)$). However, it is not difficult to see that we can define $\cP_n$ and $\cQ_n$ on the same probability space such that the joint distributions of $( n^{1/d}\cP_n, n^{1/d} \cQ_n)$ and $(\cP(n),\cQ(n))$ are equal (modulo the shift by $e_d/2$). To see this define
\begin{align}\begin{split}\label{E:CouplingAlt}
		&		\cP_n \coloneqq \{ x: (x,t)\in\cP, 0\le t \le n \kappa(x) \} \\
		& \text{ and }
				\cQ_n \coloneqq \{ x: (x,t)\in\cP, 0\le t \le n \nu(x) \}.
\end{split}
\end{align}
The joint distribution of $(\cP(n),\cQ(n))$ is determined by the random variables $(\cP(n)(A)$, $\cQ(n)(B))$, where $A,B$ are Borel sets of $\R^d$. The same holds for $(n^{1/d}\cP_n,n^{1/d}\cQ_n)$. Using the independence property of the Poisson process, it is sufficient to consider the distributions of the type $[\cP(n) \setminus \cQ(n)](A)$ and $[n^{1/d}\cP_n\setminus n^{1/d}\cQ_n](A)$ or $[\cP(n) \cap \cQ(n)](A)$ and $[n^{1/d}\cP_n\cap n^{1/d}\cQ_n](A)$. Both follow the same Poisson distribution because
\begin{align*}
		\int_{  A } \int_{ \nu(x/n^{1/d}) }^{\kappa(x/n^{1/d})} \intd{t} \ \intd{x} &= \int_{ \{y: n^{1/d} y \in A\}} \int _{ \nu(x) }^{\kappa(x)} n \ \intd{t} \ \intd{y} = \int_{n^{-1/d}A} \int _{n \nu(x) }^{n \kappa(x)} \intd{t} \ \intd{y}
\end{align*}
and
\begin{align*}
				&\int_{  A } \int_{ 0 }^{\nu(x/n^{1/d}) \wedge \kappa(x/n^{1/d})} \intd{t} \ \intd{x} \\
				&= \int_{ \{y: n^{1/d} y \in A\}} \int _{ 0}^{\nu(x) \wedge \kappa(x)} n \ \intd{t} \ \intd{y} = \int_{n^{-1/d}A} \int _{0 }^{n (\nu(x) \wedge \kappa(x))} \intd{t} \ \intd{y}.
\end{align*}
For the rest of the manuscript we will use the following definitions that apply to both the Poisson and the binomial case:
\begin{align*}
		& \ol \chi_{\kappa,n} (t) \coloneqq n^{-1/2} ( \chi( \cK_{t,n} ) - \E{\chi( \cK_{t,n}  )} ),\\
		&	\ol \chi_{\nu,n} (t)  \coloneqq n^{-1/2} ( \chi( \wt \cK_{t,n} ) - \mathbb{E}[\chi( \wt \cK_{t,n} )] ),\quad\\		
		& \ol \chi_{\kappa,n,i} (t) \coloneqq n^{-1/2} ( \chi( \cK'_{t,n,i} ) - \E{\chi( \cK'_{t,n,i} )} ), \\
		& \ol \chi_{\nu,n,i} (t)  \coloneqq n^{-1/2} ( \chi( \wt \cK'_{t,n,i} ) - \mathbb{E}[\chi(\wt \cK'_{t,n,i} )] ),
\end{align*}
for $t\in [0,T]$ and $i \in [{b_n^\prime}]$. 

\begin{lemma}\label{L:MomentsEuler}
Let $m \in \N_0$. There is a constant $C \in\R_+$ not depending on $m$ such that
\[
		\binom{m}{k} = \frac{m!}{k!(m-k)!} \le C  \frac{2^m}{m^{1/2} },  \quad \forall k\in\{0,\ldots,m \}.
\]
In particular, if $\lambda\in\R_+$ and $X \sim \poi(\lambda)$, then $\mathbb{E}[ |\sum_{k=0}^X \binom{X}{k} |^q ] < \infty$ for all $q\in\R_+$.
\end{lemma}
\begin{proof}
The result relies on the Stirling formula
$$
			\sqrt{2\pi} n^{n+1/2} e^{-n} \le n! < e n^{n+1/2} e^{-n} \text{ for } n\in\N.
$$		
It is well-known that the binomial coefficient is maximal at $m/2$ if $m$ is even and at $(m+1)/2$ if $m$ is odd. Thus, if $m$ is even,
\begin{align*}
		 \frac{m!}{k!(m-k)!} &\le  \frac{m!}{((m/2)!)^2} \le \frac{m! \ 2^{m+1} }{2\pi \ m^{m+1} \ e^{-m} } \le \frac{e \ 2^{m} }{\pi \ m^{1/2} }.
\end{align*}
A similar result is valid if $m$ is odd. The claim regarding the moment of the Poisson random variable follows immediately because $\mathbb{E}[ e^{\delta X } ] = \exp( \lambda(e^\delta -1 ) )$ is finite for all $\delta< \infty$. This completes the proof.
\end{proof}

\begin{lemma}[Bounded moments condition]\label{L:BoundedMomentsCondition}
Let $\rho > 0$, $p\in\N$ and  $\nu \in B_\infty(\kappa,\rho)$. Then there is a constant $C_p\in\R_+$ depending on $\rho$ but not on $\nu$ such that
\begin{align*}
		\sup_{n\in\N_+}\sup_{i \in [{b_n^\prime}] } \ \sup_{t\in [0,T]}  \E{ | \chi( \wt \cK_{t,n}  ) - \chi(  \wt \cK'_{t,n,i} ) |^p } \le C_p < \infty
\end{align*}
in the Poisson and in the binomial sampling scheme and for both the {\v C}ech and the Vietoris-Rips complex.
\end{lemma}
\begin{proof}[Proof of Lemma~\ref{L:BoundedMomentsCondition}]
Let $\mY,\mZ$ be two point clouds. Then
\begin{align*}
		&| \chi( \cK_t(\mY)) - \chi(\cK_t(\mZ)) | \\
		 &\le \sum_{k \le \# \mY -1 } \#\{ \sigma \in \cK_t(\mY): \sigma \text{ is a $k$-simplex intersecting with } \mY \setminus \mZ \} \\
		&\quad + \sum_{k \le \# \mZ-1 } \#\{ \sigma \in \cK_t(\mZ): \sigma \text{ is a $k$-simplex intersecting with } \mZ \setminus \mY \}. 
\end{align*}
Consequently, it suffices to study the expression
\begin{align}\label{E:BoundedMomentsCondition1}
		\E{ \Big( \sum_{k \le \# \mY -1 } \#\{ \sigma \in \cK_t(\mY): \sigma \text{ is a $k$-simplex intersecting with } \mY \setminus \mZ \} \Big)^p }.
\end{align}
We put $\mY = \cQ(n), \mZ =  \cQ'^{\,i}(n) = (\cQ(n) \setminus Q(z_{n,i}) ) \cup ( \cQ'(n)\cap Q(z_{n,i})) $ in the Poisson case, and $\mY = n^{1/d} \mY_n, \mZ = n^{1/d}\mY'^{\,i} = n^{1/d} ( \{Y'_i\} \cup [\mY_n \setminus \{Y_i \}] )$ in the binomial case, where $i$ is a generic index. (Exchanging the roles of $\mY$ and $\mZ$, the following arguments stay the same.)

In the Poisson case, the result is an immediate consequence of Lemma~\ref{L:MomentsEuler}. Indeed, let $\delta>0$ be defined as in the beginning of Section~\ref{Section_TechnicalResults}. Let $U$ be a Poisson random variable with mean $ \ |Q(0)^{(\delta)}| \ (\sup \kappa+\rho)$. Then there is a constant such that \eqref{E:BoundedMomentsCondition1} is at most
\begin{align*}
		&\E{ \Big( \sum_{k\in\N_0} \# \{ \text{$k$-simplices } \sigma\in \cK_T ( \cQ(n)): \sigma \cap Q(z_{n,i} )\neq \emptyset		\} \Big)^p } \\
		&\le C \  \sum_{m\in\N} \p( U = m ) \left ( \sum_{k=0}^{m-1} \binom{m}{k+1} \right )^p  \le C_p < \infty,
\end{align*}
where we use that conditional on $m$ there are at most $\binom{m}{k+1}$ possible $k$-simplices, so the last result follows from Lemma~\ref{L:MomentsEuler}. Clearly, the constant $C_p$ is independent of $t\in [0,T]$, $z\in\Z^d$, $n$ and $\nu$ as long as $\| \nu - \kappa \|_\infty \le \rho$.

In the binomial case, the reasoning is quite similar and we can use that the number of points is deterministic. Conditional on the realization $n^{1/d} Y'_i = x$, \eqref{E:BoundedMomentsCondition1} amounts to
\begin{align*}
    		&\E{ \Big( \sum_{k=0}^{n-1} \#\Big\{ \sigma \in \cK_t( \{x\} \cup [n^{1/d} \mY_n  \setminus \{Y_i \} ] )	: \sigma \text{ is a $k$-simplex intersecting with } \{x\}	\Big\} \Big) ^p } \\
		&\le    \E{ \Big ( \sum_{k=0}^{n-1} \sum_{  \{j_1,\ldots, j_k \} \subseteq  [n] \setminus \{i\}   } \1{ r(\{ x,Y_{i_1},\ldots,Y_{i_k}  \} ) \le n^{- 1/d} T } \Big)^p  }  \\
		&\le  \left ( \sum_{k=0}^{n-1} \sum_{ \substack{ \{j_1,\ldots, j_k\} \subseteq [n] \setminus \{i\}  } }  \left( \frac{ C(d,T) ( \| \kappa \|_{\infty} + \rho ) }{n}  \right)^k \right )^p \\
		& \le \left ( \sum_{k=0}^{n-1} \binom{n-1}{k} \left( \frac{ C(d,T) ( \| \kappa \|_{\infty} + \rho )  }{n}  \right)^k \right )^p,
\end{align*}
where the constant $C(d,T)$ depends on $d,T$ but not on the location $x$. For the last sum we have
\begin{align*}
		&   \sum_{k=0}^{n-1} \frac{(n-1)!}{ (n-1-k)!k!}  \left( \frac{ C(d,T) (\| \kappa \|_{\infty} + \rho) }{n}  \right)^k  \\
		 & \le   \sum_{k=0}^{n-1} \frac{(n-1)^k}{ n^k }  \frac{ (C(d,T)  ( \| \kappa \|_{\infty} + \rho ) )^k }{ k! } \le \exp( C(d,T) ( \| \kappa \|_\infty + \rho ) ).
\end{align*} 
\end{proof}

\subsection{Continuity properties of the filtration time}
The following properties are crucial for the desired tightness results of the EC which we will prove below. We begin with the Vietoris-Rips complex, see also \cite{thomas2019functional} for a similar result for this filtration.
 
\begin{lemma}[Continuity in the Vietoris-Rips filtration]\label{L:ContinuityVR}
Let $q\ge 1$ and $Z_0,Z_1,\ldots,Z_q$ be independent and identically distributed on $B(0,1)$ with density $\kappa$. 

Let $r_{\cR}(\{ Z_0,Z_1,\ldots,Z_q \})$ be the filtration time of the simplex $\{ Z_0,Z_1,\ldots,Z_q \}$ in the Vietoris-Rips filtration. Then for all $0 \le a\le b<\infty$
\begin{align}\label{E:ContinuityVR0}
		\p( r_{\cR}(\{ Z_0,Z_1,\ldots,Z_q \}) \in (a,b] ) \le \alpha_d \| \kappa \|_\infty  \ q(q+1) \ ( b^d - a^d).
\end{align}
\end{lemma}

\begin{proof}
We use that $\{ r_{\cR}(\{ Z_0,Z_1,\ldots,Z_q \}) \in (a,b] \}$ is contained in the finite union $\cup_{i\neq j}\{ r_{\cR}(\{Z_i,Z_j\})\in (a,b] \}$ to deduce that the left-hand side of \eqref{E:ContinuityVR0} is at most $(q+1)q \ \p(  r_{\cR}(\{Z_0,Z_1\})\in (a,b]  )$ where the last probability is at most $\| \kappa \|_\infty (| B_d(0,b) | - |B_d(0,a) |)$. Noting that the Lebesgue measure of $B_d(0,b)$ equals $\alpha_d \ b^d$, yields the result. 
\end{proof}

The next lemma gives the corresponding continuity properties in the {\v C}ech filtration.

\begin{lemma}[Continuity in the {\v C}ech filtration]\label{L:ContinuityCech}
Let $q\ge 1$ and $Z_0,Z_1,\ldots,Z_q$ be independent and identically distributed on $B(0,1)$ with density $\kappa$. Let $r_{\cC}(\{ Z_0,Z_1,\ldots,Z_q \})$ be the filtration time of the simplex $\{ Z_0,Z_1,\ldots,Z_q \}$ in the {\v C}ech filtration. Then there is a continuous function $g^*_d$ only depending on $d$,  such that for all $0\le a \le b<\infty$
\begin{align*}
		&\p( r_{\cC}(\{ Z_0,Z_1,\ldots,Z_q \}) \in (a,b] ) \\
		&\qquad \le (q+1)^{d+2}\ \max_{1\le m \le (d\wedge q)+1} (\|\kappa\|_\infty \ \alpha_{d}  )^{m} \ \cdot \ \int_a^b g^*_d(t) \intd{t}.
\end{align*}
\end{lemma}
\begin{proof}
To ease the notation we write $r$ for the filtration time and begin as follows. Let $z=(z_0,z_1,\ldots,z_q)\in B(0,1)^{q+1}$ be $(q+1)$ points in general position on $B(0,1)$. The circumsphere is the smallest $d$-dimensional ball containing all elements of $z$; its center is the circumcenter. Let $I\subset J \coloneqq \{0,1,\ldots,q\}$ and write $r[I](z)$ for the filtration time of the simplex $\{z_i:i\in I\}$.

Then the following observation is crucial: For almost every $z\in B(0,1)^{q+1}$, there are $\{i_0,i_1,\ldots,i_m\} \subset \{0,1,\ldots,q\}$ with $m\le (q+1) \wedge (d+1)$ such that 
$
	r(z) = r(\{z_{i_0}, z_{i_1},\ldots, z_{i_m}\}),
$
and $\{i_0,i_1,\ldots,i_m\} \subset \{0,1,\ldots,q\}$ is minimal w.r.t.\ inclusion (i.e. the filtration time of each strict subset of $\{i_0,i_1,\ldots,i_m\}$ is smaller). Indeed, the circumsphere and the circumcenter of $q+1$ points in general position in ${\mathbb R}^d$ is determined by at most $d+1$ points. Thus, depending on $z$, we find such a subset with cardinality at most $d+1$. In particular, given $d$ and $q$ the number of these minimal index sets is bounded above by
\[
		L\coloneqq L(d,q) \coloneqq 	\binom{q+1}{2} + \ldots + \binom{q+1}{(d+1)\wedge (q+1)} \le (q+1)^{d+2}.
\]
In the following, let $\{i_{j,0},\ldots,i_{j,m_j}\}$, $j\in [L]$, be an enumeration of these subsets.

The above insights allow us to construct the following upper bound given an arbitrary density function $\kappa$ on $B(0,1)$.
\begin{align}\label{E:ContinuityCech1}
		&\p( r(\{ Z_0,Z_1,\ldots,Z_q \}) \in (a,b] ) 	\nonumber	\\
		&\le \sum_{j=1}^{L} \p( r( \{Z_{i_{j,0}} ,Z_{i_{j,1}} ,\ldots,Z_{i_{j,m_j}}  \} ) \in (a,b] ) \nonumber \\
		&\le \sum_{j=1}^L \|\kappa \|_{\infty}^{m_j+1} \int_{B(0,1)^{m_j+1} } \1{ r(\{y_0,y_1,\ldots,y_{m_j} \} ) \in (a,b] } \intd{y}.
\end{align}
Set $\mathbb{Y} = \{Y_0,Y_1,\ldots,Y_m\}$, where the $Y_i$ are iid on $B(0,1)$ following a uniform distribution. Using the result of \cite[Lemma 6.10]{divol2020representation}, $r(\mathbb{Y})$ has a bounded density w.r.t.\ the Lebesgue measure on the real line, viz.,
$$
	\alpha_{d}^{-(m+1)} \int_{ B(0,1)^{m+1} } \1{ r(\{y_0,y_1,\ldots,y_{m} \} ) \in (a,b] } \intd{y} = \int_a^b g(u;m) \ \intd{u}
$$
for a bounded density function $g(\ \cdot\,; m)$ depending on $m$ only. Consequently, \eqref{E:ContinuityCech1} is at most
$$
	 \max_{1\le m \le (d\wedge q)+1} [ \|\kappa\|_\infty \ \alpha_{d}] ^{m} \int_a^b \sum_{j=1}^L g(u; m_j) \ \intd{u}.
$$
Setting $g^*_d = \max\{ g(\ \cdot \ ; m): m \in [d] \}$ yields the result.\end{proof}

\subsection{Approximation properties}

\begin{proof}[Proof of Theorem~\ref{T:ApproximationEuler}] Let $t\in [0,T]$. Using martingale differences and the definition of the $\sigma$-fields from \eqref{DefSigmaField}, we obtain
\begin{align}
			&\V( \ol \chi_{\kappa,n} (t) - \ol \chi_{\nu,n}(t) ) \nonumber \\
			&=  n^{-1} \sum_{i=1}^{{b_n^\prime}}  \E{ \Big( \E{ \chi_{\kappa,n}(t)  - \chi_{\nu,n}(t)  | \cG_{n,i} } - \E{ \chi_{\kappa,n}(t)  - \chi_{\nu,n}(t)  | \cG_{n,i-1} } \Big) ^2  	} \nonumber \\
			&= n^{-1} \sum_{i=1}^{{b_n^\prime}} \E{ \E{ \chi_{\kappa,n}(t)  - \chi_{\nu,n}(t) - \chi_{\kappa,n,i}(t) + \chi_{\nu,n,i}(t) \ | \ \cG_{n,i} }^2  	} \nonumber \\ 
			&\le n^{-1} \sum_{i=1}^{{b_n^\prime}} \E{  \Big( \chi_{\kappa,n}(t)  - \chi_{\nu,n}(t) - \chi_{\kappa,n,i}(t) + \chi_{\nu,n,i}(t)  \Big)^2 } 	. \label{E:ApproximationEuler1}
\end{align}
For the summands in \eqref{E:ApproximationEuler1} we have by using the definition of the EC that
\begin{align}
			& \E{  \Big(\sum_{k\in \N_0} (-1)^k \Big( S_k(\cK_{t,n} ) - S_k(\cK'_{t,n,i} ) -  S_k(\wt \cK_{t,n} ) + S_k(\wt \cK'_{t,n,i} )		\Big)		\Big)^2	 }\nonumber \\ 
			&\le 3 \E{	\Big(  \sum_{k\in \N_0} \Big| S_k(\cK_{t,n} ) - S_k(\cK'_{t,n,i} ) -  S_k(\wt \cK^*_{t,n,i} ) + S_k(\wt \cK'^*_{t,n,i} )		\Big| \Big)^2	}	\label{E:ApproximationEuler2} \\ 
			\begin{split}
			&\quad + 3 \E{ \Big( \sum_{k\in \N_0}  \Big|S_k(\wt \cK_{t,n} ) -  S_k(\wt \cK^*_{t,n,i} )	\Big|	\Big)^2} \\
			&\quad + 3 \E{	 \Big( \sum_{k\in \N_0}  \Big| S_k(\wt \cK'_{t,n,i} ) -  S_k(\wt \cK'^*_{t,n,i} )		\Big|	\Big)^2	}.  \label{E:ApproximationEuler3} 
			\end{split}
\end{align}
With $\epsilon \coloneqq \| \kappa - \nu \|_{\infty}$, it remains to show that each of the last three expectations is of order $\epsilon$ uniformly in $i\in [{b_n^\prime}], n$ and $t \in [0,T]$.  For this, we consider the two sampling schemes separately.

\textit{The Poisson case.} We define $\ol W(n)$ as the union of $\cP(n),\cQ(n), \cP'(n),\cQ'(n)$. Also define symmetric differences of Poisson processes: 
\[
	W_{n,\epsilon} = \cP(n) \triangle \cQ(n) \text{ and } W'_{n,\epsilon} =   \cP'(n) \triangle \cQ'(n).
	\]
	When restricted to a cube $Q=Q(z_{n,i})$, $W_{n,\epsilon}$ and $W'_{n,\epsilon}$ are not empty with a probability of order at most $\epsilon$. Clearly, given a point $Z$ in $W_{n,\epsilon}$, this point can only be involved in simplices which lie inside the $\delta$-neighborhood $Q^{(\delta)}$ of $Q$, where $\delta$ is the upper bound on the diameter of the simplices (defined at the beginning of Section 4), which is only depending on $T$ and $d$ but neither on $n$ nor on $i \in [b'_n].$
	
First we consider \eqref{E:ApproximationEuler3}, here we give the details for the first term only; the second term can be treated very similarly. Write $\ol W(n) = W_{n,\epsilon} \cup W'_{n,\epsilon} \cup W(n)$, where $W(n)$ is the Poisson process that collects all remaining points from $\ol W(n)\setminus ( W_{n,\epsilon} \cup W'_{n,\epsilon})$, so it has a finite intensity. Then $ | S_k(\wt \cK_{t,n,i} ) -  S_k(\wt \cK^*_{t,n,i} )|	$ is stochastically dominated by the random variable
\begin{align}\label{E:ApproximationEuler4} 
&\sum_{ Z \in W_{n,\epsilon} \cap Q } \sum_{   \substack{(Y_1,\ldots, Y_k ) \subseteq \ol W(n) \cap Q^{(\delta)} \\ Y_i \neq Y_j }  } \1{ r( \{	Z,Y_1,\ldots,Y_k \} ) \le T } \nonumber \\
& \le W_{n,\epsilon}(Q) \binom{ \ol W(n)(Q^{(\delta)} )}{k} \nonumber\\
&\le  C \1{ W_{n,\epsilon} (Q^{(\delta)}) > 0 } \sqrt{ \ol W(n)(Q^{(\delta)} ) } \ 2^{ \ol W(n)(Q^{(\delta)} ) },
\end{align}
where the last inequality follows from Lemma~\ref{L:MomentsEuler}, and by bounding $W_{n,\epsilon}(Q)$ by $\ol W(n)(Q^{(\delta)})$.

We can compute moments of this last expression by exploiting the independence between $(\cP,\cQ)$ and $(\cP',\cQ')$. Indeed, the components $W_{n,\epsilon}$, $W'_{n,\epsilon}$ and $W(n)$ are independent, and so it is sufficient to consider (for $\wt C\in\R_+$)
\begin{align*}
		\E{  \1{ W_{n,\epsilon} (Q^{(\delta)}) > 0 } \ {\wt C}^{  W_{n,\epsilon} (Q^{(\delta)} ) } } &= \sum_{k=1}^{\infty} \p(  W_{n,\epsilon} (Q^{(\delta)}) = k ) \  {\wt C}^{ k} \\
		& \le C_1 (1 - e^{-C_2 \epsilon} ) \le C_3 \epsilon,
 \end{align*}
for constants $C_1,C_2,C_3 < \infty$; the last inequality follows by the mean-value theorem. This completes the considerations for \eqref{E:ApproximationEuler3}.

Second, we study the term in \eqref{E:ApproximationEuler2}. This second order difference 
$$
	 | S_k(\cK_{t,n} ) - S_k(\cK'_{t,n,i} ) -  S_k(\wt \cK^*_{t,n,i} ) + S_k(\wt \cK'^*_{t,n,i}) |	
$$
 can only be non zero if $W_{n,\epsilon}(Q^{(\delta)})>0$. The conclusion follows now in a similar fashion as before (see \eqref{E:ApproximationEuler4}). Reasoning as we did after \eqref{E:CouplingAlt}, we deduce \eqref{E:ApproximationEuler0}.

\textit{The binomial case.} The structure of the proof works in the same fashion. Since after applying the decomposition in martingale differences both  \eqref{E:ApproximationEuler2} and  \eqref{E:ApproximationEuler3} do not depend on the $\sigma$-field $\cG_{n,i}$,  we consider w.l.o.g. the case $i=1$ (this simplifies the notation). We begin with the first term in \eqref{E:ApproximationEuler3}:
\begin{align*}
			& \sum_{k=0}^{n-1} \Big | S_k( \cK_t (n^{1/d} \{Y_1, Y_2,\ldots, Y_n \}))  - S_k( \cK_t ( n^{1/d} \{ X_1,Y_2, \ldots, Y_n \}) ) \Big| \\
			& \le \1{Y_1 \neq X_1 } \ \sum_{k=0}^{n-1} \sum_{ i_1,\ldots, i_k} \1{  r(\{ Y_1, Y_{i_1}, \ldots, , Y_{i_k} 		\}) \le n^{- 1/d} \ T	} \\
			&\quad \qquad\qquad\qquad\qquad\qquad\qquad +  \1{  r(\{ X_1, Y_{i_1}, \ldots, , Y_{i_k} 		\}) \le n^{- 1/d} \ T	},
\end{align*}
where the second sum is taken over all combinations $(i_1,\ldots,i_k)$ in $\{2,\ldots,n\}$ with pairwise different indices. 

Also, for $k_1,k_2 \le n-1$ and for two sets $\{i_1,\ldots, i_{k_1} \}, \{j_1,\ldots, j_{k_2} \}$, which have $\ell$ common elements, the probabilities
\begin{align*}
	&	\p( r(\{ Y_1, Y_{i_1}, \ldots, , Y_{i_{k_1} } 		\}) \le n^{-1/d} \ T ,  r(\{ X_1, Y_{j_1}, \ldots , Y_{j_{k_2} } 		\}) \le n^{-1/d} \  T \ | \ X_1, Y_1 ) \\
	& \text{and }\\
	&  \p( r(\{ Y_1, Y_{i_1}, \ldots, , Y_{i_{k_1} } 		\}) \le n^{-1/d} \ T ,  r(\{ Y_1, Y_{j_1}, \ldots , Y_{j_{k_2} } 		\}) \le n^{-1/d} \ T \ | \ X_1, Y_1 )
\end{align*}
are at most $(A/n)^{k_1+k_2-\ell}$ for a constant $A \ge C_{d,T} \ ( \|\kappa\|_\infty + \rho )$, where the constant $C_{d,T}$ only depends on $d$ and $T$.
Using this last insight, elementary combinatorial calculations unveil
\begin{align*}
	&\E{ \Big( \sum_{k=0}^{n-1} \Big | S_k( \cK_t ( n^{1/d} \{Y_1, Y_2,\ldots, Y_n \}))  - S_k( \cK_t(n^{1/d} \{ X_1,Y_2, \ldots, Y_n \}))) \Big|\;	\Big)^2 } \\
	&\le \epsilon \ \sum_{k_1,k_2 = 0}^{n-1} \sum_{\ell=0}^{k_2 \wedge k_1}  \binom{ n-1 }{\ell} \ \binom{ n-1-\ell }{ k_1 - \ell } \ \ \binom{ n-1-k_1 }{ k_2 - \ell } \Big( \frac{A}{n} \Big) ^{k_1 + k_2 - \ell } \\
	&= \epsilon \ \sum_{k_1,k_2 = 0}^{n-1} \sum_{\ell=0}^{k_1 \wedge k_2}  \frac{(n-1)!}{ (n-1-k_2-k_1+\ell )! \ n^{k_1+ k_2 -\ell} } \frac{A^{\ell} }{\ell!} \frac{A^{k_1 - \ell} }{(k_1 - \ell) !} \frac{A^{k_2 - \ell} }{(k_2 - \ell) !} \\
	&\le  \epsilon \ \sum_{k_1,k_2 = 0}^{n-1} \sum_{\ell=0}^{k_1 \wedge k_2}  \frac{A^{\ell} }{\ell!} \frac{A^{k_1 - \ell} }{(k_1 - \ell) !} \frac{A^{k_2 - \ell} }{(k_2 - \ell) !} \le C \epsilon,
\end{align*}
for a constant $C< \infty$.

The bound for the second order difference in \eqref{E:ApproximationEuler2} follows in a similar fashion; we omit most of the  details here, but refer to the proof of Theorem~\ref{T:FunctionalWasserstein}, where second order differences are studied in great detail. We have, 
\begin{align*}
		& \Big|  S_k( \cK_t( n^{1/d}	\{ X_1,X_2,\ldots, X_n		\}	) )  - S_k( \cK_t(n^{1/d}	\{ X'_1,X_2,\ldots, X_n		\} )	) \\
		&\qquad - S_k( \cK_t(	n^{1/d}\{ X_1,Y_2,\ldots, Y_n	\} )	)	+  S_k( \cK_t(	n^{1/d}\{ X'_1,Y_2,\ldots, Y_n		\}	) )	\Big| \\
&= \Big| \sum_{i_1,\ldots,i_k} \Big( \1{ r( \{X_1,X_{i_1},\ldots, X_{i_k} 	\} ) \le n^{-1/d} \ t } \\
&\quad\qquad\qquad  - \1{ r( \{X'_1,X_{i_1},\ldots, X_{i_k} 	\} ) \le n^{-1/d} \ t } \Big) \\
&\quad\qquad -  \sum_{i_1,\ldots,i_k} \Big( \1{ r( \{X_1,Y_{i_1},\ldots, Y_{i_k} 	\} ) \le n^{-1/d} \ t }  \\
&\quad \qquad\qquad -  \1{ r( \{X'_1,Y_{i_1},\ldots, Y_{i_k} 	\} ) \le n^{-1/d} \ t } \Big) \Big | \\
&\le   \sum_{i_1,\ldots,i_k}   \Big |  \1{ r( \{X_1,X_{i_1},\ldots, X_{i_k} 	\} ) \le n^{-1/d} \ t } \\
&\quad\qquad \qquad - \1{ r( \{X_1,Y_{i_1},\ldots, Y_{i_k} 	\} ) \le n^{-1/d} \  t }  \Big |   \\
&\quad + \sum_{i_1,\ldots,i_k} \Big | \1{ r( \{X'_1,Y_{i_1},\ldots, Y_{i_k} 	\} ) \le n^{-1/d} \ t } \\
&\quad\qquad\qquad -  \1{ r( \{X'_1,X_{i_1},\ldots, X_{i_k} 	\} ) \le n^{-1/d} \ t } \Big |,
\end{align*}
where the sums are taken over all $\binom{n-1}{k}$ $k$-element subsets $(i_1,\ldots,i_k)$ of $\{2,\ldots,n\}.$ 
Clearly, a term only contributes to the sum if there is at least one index $u \in \{i_1,\ldots, i_k\}$ for which $X_u \neq Y_u$.  By using similar arguments as in the proof of Theorem~\ref{T:FunctionalWasserstein}), we arrive at:
\begin{align*}
		& \mathbb{E} \Big[ \Big( \sum_{k=0}^{n-1} \Big|  S_k( \cK_t(  n^{1/d} 	\{ X_1,X_2,\ldots, X_n		\}	)  ) - S_k( \cK_t(  n^{1/d}  	\{ X'_1,X_2,\ldots, X_n		\}	)  )\\
		&\quad - S_k( \cK_t( n^{1/d} 	\{ X_1,Y_2,\ldots, Y_n	\}	) )	+  S_k( \cK_t(	n^{1/d}  \{ X'_1,Y_2,\ldots, Y_n		\}	) )	\Big| \ \Big)^2 \Big] \le C \epsilon.
\end{align*}
This yields \eqref{E:ApproximationEuler0}. The rate of convergence in the Kantorovich-Wasserstein distance is an immediate consequence.
\end{proof}

\begin{proof}[Proof of Theorem~\ref{T:FunctionalWasserstein}]
The proof has the same structure as the proof of Theorem~\ref{T:ApproximationEuler}. Let $\nu$ be arbitrary but fixed with $\epsilon \coloneqq \|\kappa - \nu \|_\infty \le \rho$. Let $[0,T]$ be partitioned into $J$ equidistant intervals of length $T/J$ marked by the points $t_0,t_1,\ldots,t_J$. We apply the fundamental decomposition
\begin{align}
		&\E{ \sup_{t\in [0,T]} |\ol \chi_{\kappa,n}(t) - \ol \chi_{\nu,n}(t) |^2 } \nonumber \\
		\begin{split}\label{E:FunctionEulerApproximation1}
		&\le 2 \ J \ \max_{i \le J} \E{  |\ol \chi_{\kappa,n}(t_i) - \ol \chi_{\nu,n}(t_i) |^2 } \\
		&\quad + 2 \ J \ \max_{i \le J} \E{ \sup_{t\in [t_{i-1},t_i] } |\ol \chi_{\kappa,n}(t) - \ol \chi_{\nu,n}(t) - \ol \chi_{\kappa,n}(t_i) + \ol \chi_{\nu,n}(t_i)|^2 }.
\end{split}\end{align}
The first term in \eqref{E:FunctionEulerApproximation1} can be treated with the result in \eqref{E:ApproximationEuler0}. In order to obtain a bound on the second term, we use the monotonicity of the EC. It is enough to study a generic index $i\in [J]$, so we set $a=t_{i-1}$ and $b=t_i$. Also, we write $t^*$ for the time in $[a,b]$, where the supremum is attained. So $t^*$ is random and measurable.

First, we decompose the EC in two terms which contain the simplices of even, resp.\ odd, dimension. So, the first term is indexed by $I_1 = \{ k \in \N_0: k \text{ is even} \}$, the second by $I_2 = \N_0\setminus I_1$. We only consider the index set $I_1$, $I_2$ works in a similar fashion. The part of the second term in \eqref{E:FunctionEulerApproximation1}, which is related to $I_1$, is then
\begin{align}\begin{split}\label{E:FunctionEulerApproximation2}
		&n^{-1} \ \mathbb{E} \Bigg[ \Big( \sum_{k\in I_1} S_k( \cK_{t^*,n} ) - S_k( \wt \cK_{t^*,n} ) - S_k( \cK_{b,n} ) - S_k( \wt \cK_{b,n} ) \\
		&\qquad \qquad  - {\mathbb E}\Big[\sum_{k\in I_1}  S_k( \cK_{t^*,n} ) - S_k( \wt \cK_{t^*,n} ) - S_k( \cK_{b,n} ) - S_k( \wt \cK_{b,n} ) \Big]  \Big)^2\Bigg].
\end{split}\end{align}
We have for a specific dimension $k\in \N_0$
\begin{align*}
		& S_k( \cK_{t^*,n} ) - S_k( \wt \cK_{t^*,n} ) - S_k( \cK_{b,n} ) + S_k( \wt \cK_{b,n} ) \\
		&= \sum_{ \substack{ \sigma \in \wt \cK_{b,n} \setminus \cK_{b,n} , \\ \dim(\sigma) = k } } \1{ r(\sigma) \in (t^*,b] } \qquad  - \sum_{ \substack{ \sigma \in \cK_{b,n} \setminus \wt \cK_{b,n}, \\ \dim(\sigma) = k } } \1{ r(\sigma) \in (t^*,b] } ,
\end{align*}
all other simplices cancel. Using this insight, we split \eqref{E:FunctionEulerApproximation2} in two terms as follows
\begin{align}\begin{split}\label{E:FunctionEulerApproximation3}
		&n^{-1} \ \mathbb{E} \Bigg[ \Bigg| \sum_{k\in I_1} \sum_{ \substack{ \sigma \in \cK_{b,n} \setminus \wt \cK_{b,n} , \\ \dim(\sigma) = k } } \1{ r(\sigma) \in (t^*,b] } \\
		&\qquad\qquad\qquad - \mathbb{E}\Big[ \sum_{k\in I_1} \sum_{ \substack{ \sigma \in \cK_{b,n} \setminus \wt \cK_{b,n} , \\ \dim(\sigma) = k } } \1{ r(\sigma) \in (t^*,b] } \Big]  \Bigg|^2 \Bigg], \\
		&n^{-1} \ \mathbb{E} \Bigg[ \Bigg| \sum_{k\in I_1} \sum_{ \substack{ \sigma \in \wt \cK_{b,n} \setminus \cK_{b,n} , \\ \dim(\sigma) = k } } \1{ r(\sigma) \in (t^*,b] }  \\
		&\qquad\qquad\qquad  - \mathbb{E}\Big[ \sum_{k\in I_1} \sum_{ \substack{ \sigma \in \wt \cK_{b,n} \setminus \cK_{b,n} , \\ \dim(\sigma) = k } } \1{ r(\sigma) \in (t^*,b] } \Big] \Bigg|^2 \Bigg].
\end{split}\end{align}
Clearly, it is enough to study the first term \eqref{E:FunctionEulerApproximation3}. If it is positive, then the double sum in the first term in \eqref{E:FunctionEulerApproximation3} is at most
\begin{align}
	&\sum_{k\in I_1} \sum_{ \substack{ \sigma \in \cK_{b,n} \setminus \wt \cK_{b,n} , \\ \dim(\sigma) = k } } \1{ r(\sigma) \in (a,b] }  - \mathbb{E}\Bigg[ \sum_{k\in I_1} \sum_{ \substack{ \sigma \in \cK_{b,n} \setminus \wt \cK_{b,n} , \\ \dim(\sigma) = k } } \1{ r(\sigma) \in (a,b] } \Bigg] \nonumber \\
	&\quad + \mathbb{E}\Bigg[ \sum_{k\in I_1} \sum_{ \substack{ \sigma \in \cK_{b,n} \setminus \wt \cK_{b,n} , \\ \dim(\sigma) = k } } \1{ r(\sigma) \in (a,b] } - \1{ r(\sigma) \in (t^*,b] } \Bigg] \nonumber\\
	\begin{split}\label{E:FunctionEulerApproximation4}
	&\le \Bigg| \sum_{k\in I_1} \sum_{ \substack{ \sigma \in \cK_{b,n} \setminus \wt \cK_{b,n} , \\ \dim(\sigma) = k } } \1{ r(\sigma) \in (a,b] }  - \mathbb{E}\Bigg[ \sum_{k\in I_1} \sum_{ \substack{ \sigma \in \cK_{b,n} \setminus \wt \cK_{b,n} , \\ \dim(\sigma) = k } } \1{ r(\sigma) \in (a,b] } \Bigg] \Bigg| \\
	&\quad + \mathbb{E}\Bigg[ \sum_{k\in I_1} \sum_{ \substack{ \sigma \in \cK_{b,n} \setminus \wt \cK_{b,n} , \\ \dim(\sigma) = k } } \1{ r(\sigma) \in (a,b] } \Bigg]  .
	\end{split}
\end{align}
Otherwise, if it is negative, then the double sum in the first term in \eqref{E:FunctionEulerApproximation3} is at most
\begin{align*}
		& \mathbb{E}\Bigg[ \sum_{k\in I_1} \sum_{ \substack{ \sigma \in \cK_{b,n} \setminus \wt \cK_{b,n} , \\ \dim(\sigma) = k } } \1{ r(\sigma) \in (a,b] } \Bigg],
\end{align*}
and this term is already contained in the estimate in \eqref{E:FunctionEulerApproximation4}. Hence, it is enough to derive upper bounds for
\begin{align}
		&n^{-1} \mathbb{E}\Bigg[  \Bigg| \sum_{k\in I_1} \sum_{ \substack{ \sigma \in \cK_{b,n} \setminus \wt \cK_{b,n} , \\ \dim(\sigma) = k } } \1{ r(\sigma) \in (a,b] }  - \mathbb{E}\Big[ \sum_{k\in I_1} \sum_{ \substack{ \sigma \in \cK_{b,n} \setminus \wt \cK_{b,n} , \\ \dim(\sigma) = k } } \1{ r(\sigma) \in (a,b] } \Big]  \Bigg|^2 \Bigg]\label{E:FunctionEulerApproximation5}\\
		&\text{ and } \quad n^{-1} \mathbb{E}\Bigg[ \sum_{k\in I_1} \sum_{ \substack{ \sigma \in \cK_{b,n} \setminus \wt \cK_{b,n} , \\ \dim(\sigma) = k } } \1{ r(\sigma) \in (a,b] } \Bigg]^2 \label{E:FunctionEulerApproximation6}
\end{align}
separately. We begin with \eqref{E:FunctionEulerApproximation5}, which can be treated with an MDS approach similar as in the proof of Lemma~\ref{L:BoundedMomentsCondition} and Theorem~\ref{T:ApproximationEuler}. By using this MDS approach we obtain that the expression in \eqref{E:FunctionEulerApproximation5} is at most
\begin{align}
			&n^{-1} \sum_{i=1}^{{b_n^\prime}} \mathbb{E}\Bigg[ \Bigg| \sum_{k\in I_1} \sum_{ \substack{ \sigma \in \cK_{b,n} \setminus \wt \cK_{b,n} , \\ \dim(\sigma) = k } } \1{ r(\sigma) \in (a,b] }  - \sum_{ \substack{ \sigma' \in \cK'_{b,n,i} \setminus \wt \cK'_{b,n,i} , \\ \dim(\sigma) = k } } \1{ r(\sigma') \in (a,b] }  \Bigg|^2 \Bigg] \label{boundE:FunctionEulerApproximation5}
\end{align}
We continue by estimating (\ref{boundE:FunctionEulerApproximation5}) and (\ref{E:FunctionEulerApproximation6})  in the Poisson and the binomial case separately.\\

\textit{The Poisson case.} Using the fact that the simplices $\sigma$ with $\sigma \cap Q(z_{n,i}) = 0$ appear in both of the double sums inside the expectation in (\ref{boundE:FunctionEulerApproximation5}), we can bound \eqref{boundE:FunctionEulerApproximation5} by
\begin{align}
		\begin{split}\label{E:FunctionEulerApproximation7}
		& 2 \ n^{-1} \sum_{i=1}^{{b_n^\prime}} \mathbb{E}\Bigg[ \Bigg| \sum_{k\in I_1} \sum_{ \substack{ \sigma \in \cK_{b,n} \setminus \wt \cK_{b,n} , \\ \dim(\sigma) = k } } \1{\sigma\cap Q(z_{n,i}) \neq \emptyset, r(\sigma) \in (a,b] }    \Bigg|^2 \Bigg] \\
		&\qquad\qquad + 2 \ n^{-1} \sum_{i=1}^{{b_n^\prime}} \mathbb{E}\Bigg[ \Bigg| \sum_{k\in I_1}  \sum_{ \substack{ \sigma' \in \cK'_{b,n,i} \setminus \wt \cK'_{b,n,i} , \\ \dim(\sigma) = k } } \1{\sigma' \cap Q(z_{n,i}) \neq \emptyset, r(\sigma') \in (a,b] } \Bigg|^2 \Bigg].
		\end{split}
\end{align}
 Clearly, it is enough to study the first term in \eqref{E:FunctionEulerApproximation7}. We show that there is a constant, which is uniform in $i$ and $n$, such that each expectation is at most $C |b-a| \epsilon $, where $\epsilon$ is an upper bound for the supremum distance between the densities $\kappa$ and $\nu$. Let $i\in [{b_n^\prime}]$ be arbitrary but fixed and set $Q=Q(z_{n,i})$. Moreover, we let
 \begin{align*}
			W_n = \cP(n) \cap \cQ(n), \qquad 
			\cP_{n,\epsilon} = \cP(n)\setminus \cQ(n), \qquad 
			\cQ_{n,\epsilon} = \cQ(n) \setminus \cP(n). 
\end{align*}
    These processes are independent. First, we compute the expectation on the cube $Q$ given that $W_{n} (Q^{(\delta)}) = m$ and $\cP_{n,\epsilon} (Q^{(\delta)}) = \wt m$, so that we can write $\cP_{n,\epsilon} \cap Q^{(\delta)} = \{Z_1,\ldots, Z_{\wt m} \}$ and $\cP(n) \cap Q^{(\delta)} = \{Y_1,\ldots, Y_{m^*} \}$, where $m^* = m + \wt m$. Then the expectation is dominated by 
\begin{align}\label{E:FunctionEulerApproximation8}
			 \E{ \Bigg| \sum_{k=0}^{m-1} \sum_{u=1}^{\wt m} \sum_{ (i_1,\ldots,i_k) \subseteq [m^*] } \1{ r(\{Z_u, Y_{i_1},\ldots,Y_{i_k} \}) \in (a,b]} \Bigg|^2 },
\end{align}
note that this expression is 0 if $\wt m = 0$ because each simplex necessarily contains at least one Poisson point of $\cP_{n,\epsilon}$. In order to compute the expectation, we need to control
\begin{align}\label{E:FunctionEulerApproximation8a}
		\p( r(\{Z_u, Y_{i_1},\ldots,Y_{i_k} \}) \in (a,b], r(\{Z_{u'}, Y_{j_1},\ldots,Y_{j_{k'} } \}) \in (a,b] )
\end{align}
for arbitrary tuples $(i_1,\ldots,i_k), (j_1,\ldots,j_{k'})$ and indices $k,k', u,u'$. For this, we simply omit the simplex with the higher dimension. We obtain for the Vietoris-Rips filtration the uniform upper bound $C_{d,T,\|\kappa \|_\infty,\rho}  (k\wedge k')^2 |b-a|$, as in Lemma~\ref{L:ContinuityVR}. The constant $C_{d,T,\|\kappa\|_\infty,\rho}$ only depends on $d$, $T$, $\|\kappa\|_\infty,\rho$ and is independent* of $\nu$. Then \eqref{E:FunctionEulerApproximation8} is at most
\begin{align}
				& C \sum_{k,k' = 1}^{m^*} \sum_{u,u'=1}^{\wt m} \binom{ m^*}{k} \binom{ m^*}{k'}  (k \wedge k')^2 |b-a| \\
				& \le C (m^*)^p 2^{2m^*} |b-a| \1{ \wt m > 0},\label{E:FunctionEulerApproximation8b}
\end{align}
for some $p\in\R_+$, and a constant $C$ independent of $i\in [{b_n^\prime}]$, $n$ and independent* of $\nu$.

If the {\v C}ech filtration is used instead, we bound above the probability in \eqref{E:FunctionEulerApproximation8a} by $C'_{d,T,\|\kappa \|_\infty,\rho} (k\wedge k')^{d+2} |b-a|$ for a constant $C'_{d,T,\|\kappa\|_\infty,\rho}$, which only depends on $d$, $T$, $\|\kappa\|_\infty,\rho$ and which is independent* of $\nu$, see Lemma~\ref{L:ContinuityCech}. So the upper bound in \eqref{E:FunctionEulerApproximation8b} changes only in terms of the constants but not in its structure.

Finally, we have to weigh this last upper bound according to the distribution of $\cP_{n,\epsilon} (Q^{(\delta)})$ and $W_n (Q^{(\delta)})$. Note that the Poisson parameter of $\cP_{n,\epsilon} (Q^{(\delta)})$ is at most $\epsilon \ |Q^{(\delta)}|$. It is now straightforward to show that (for each $\wt C< \infty$)
\begin{align*}
		&\sum_{m\in\N_0} \sum_{\wt m\in\N_0} \Big(\p(W_n (Q^{(\delta)}) = m ) \ \p( \cP_{n,\epsilon} (Q^{(\delta)}) = \wt m) \ (m^*)^p e^{\wt C m^*} |b-a| \1{ \wt m > 0}\Big) \\
		& \le C \epsilon |b-a|
\end{align*}
for a constant $C$ that is also independent of $i\in[\kappa]$ and $n$ and independent* of $\nu$. This shows the claim for \eqref{E:FunctionEulerApproximation5} because $|b-a| = T/ J$. 

The claim regarding \eqref{E:FunctionEulerApproximation6} follows in a similar fashion by partitioning again the simplices according to their position:
\begin{align*}
		& n^{-1} \mathbb{E}\Bigg[ \sum_{i=1}^{{b_n^\prime}} \sum_{k\in I_1} \sum_{ \substack{ \sigma \in \cK_{b,n} \setminus \wt \cK_{b,n} , \\ \dim(\sigma) = k } } \1{ \sigma\cap Q(z_{n,i}) \neq \emptyset, r(\sigma) \in (a,b] } \Bigg]^2 \\
		&\le n^{-1} \left( \sum_{i=1}^{{b_n^\prime}} C \epsilon |b-a| \right)^2  \le C n |b-a|^2 \epsilon^2.
\end{align*}
This last upper bound is of order $C T^2 \epsilon^2 n J^{-2}$. 

\textit{The binomial case.} Again, we begin with \eqref{boundE:FunctionEulerApproximation5}. Here the martingale difference sequence is constructed by replacing points $X_i$ and $Y_i$ with independent points $X'_i$ and $Y'_i$,  respectively. For a fixed $i$, we have the following relation. A $k$-simplex $\sigma$ in $\cK_{b,n}$ not containing $n^{1/d} X_i$,  also lies in $\cK'_{b,n,i}$, and thus cancels in (\ref{boundE:FunctionEulerApproximation5}). The same holds for a $k$-simplex $\sigma'$ not containing $n^{1/d} X'_i$: if $\sigma' \in \cK'_{b,n,i}$, then $\sigma' \in \cK_{b,n}$. Again, these simplices cancel in (\ref{boundE:FunctionEulerApproximation5}). Clearly, this relation similarly holds for the simplicial complexes $\wt \cK_{b,n}$ and $\wt \cK_{b,n,i}$. Hence, \eqref{boundE:FunctionEulerApproximation5} is at most
\begin{align}
		\begin{split}\label{E:FunctionEulerApproximation9}
		& 2 \ n^{-1} \sum_{i=1}^{n} \mathbb{E}\Bigg[ \Bigg| \sum_{k\in I_1} \sum_{ \substack{ \sigma \in \cK_{b,n} \setminus \wt \cK_{b,n} , \\ \dim(\sigma) = k } } \1{\sigma = n^{1/d} \{X_i, X_{j_1}, \ldots, X_{j_k}\} , \{j_1 ,\ldots, j_k \} \subseteq [n] } \\
		&\qquad\qquad  \1{ r(\sigma) \in  (a,b]  , X_i \neq Y_i \text{ or } X_{j_\ell}\neq Y_{j_\ell} \text{ for at least one } \ell\in [k] } \Bigg|^2 \Bigg] \end{split}
		\\
		&\quad +2 \ n^{-1} \sum_{i=1}^{n} \mathbb{E}\Bigg[ \Bigg| \sum_{k\in I_1} \sum_{ \substack{ \sigma \in \cK_{b,n} \setminus \wt \cK_{b,n} , \\ \dim(\sigma) = k } } \1{\sigma = n^{1/d} \{X'_i, X_{j_1}, \ldots, X_{j_k}\} ,  \{j_1 ,\ldots, j_k \} \subseteq [n] }  \nonumber \\
		&\qquad\qquad  \1{ r(\sigma) \in (a,b] ,X'_i \neq Y'_i \text{ or } X_{j_\ell}\neq Y_{j_\ell} \text{ for at least one } \ell\in [k] }  \Bigg|^2 \Bigg] . \nonumber
\end{align}
It suffices to consider the expectation in \eqref{E:FunctionEulerApproximation9} for an arbitrary but fixed index $i$. Let $N$ be the number of observations $X_j$ in the $\delta$-neighborhood of $X_i$, so, 
$$
	N = \sum_{j: j\neq i} \1{ n^{1/d} X_j \in B( n^{1/d} X_i,\delta) } + 1.
$$
Conditional on $N$ and $X_i$, we can then compute the expectation. To this end, consider a generic simplex $\{ n^{1/d} X_i, Z_1,\ldots,Z_k\}$, where $Z_i$ are iid on $B(n^{1/d} X_i,\delta)$ from $\mX_n$ with a strictly positive and bounded density function. In the same fashion, we write $\wt Z_i$ for the corresponding elements from $\mY_n$. Then, for the Vietoris-Rips complex, by using Lemma~\ref{L:ContinuityVR},
\begin{align}\begin{split}\label{E:FunctionEulerApproximation9b}
	&	 \p\Big( r( \{n^{1/d}X_i, Z_1,\ldots, Z_k \} ) \in  (a,b], \\
	&\qquad\qquad X_i \neq Y_i \text{ or } Z_j \neq \wt Z_j \text{ for a } j\in \{1,\ldots,k\} \Big) \le C (k+1) k \ |b-a|  \epsilon.
		\end{split}
\end{align}
where $C$ only depends on $d$, $T$, $\|\kappa\|_\infty$ and $\rho$. If the {\v C}ech filtration is used instead, we exchange the factor $C (k+1) k$ by $k^{d+2}$ (multiplied by a certain constant which only depends on $d$, $T$, $\|\kappa\|_\infty$ and $\rho$); see Lemma~\ref{L:ContinuityCech}.

 Consequently, similar to the Poisson case, we obtain as an upper bound  for a single expectation in \eqref{E:FunctionEulerApproximation9}
\begin{align}
		&\sum_{m=1}^{n} \p(N=m) \sum_{k,k'=1}^m \binom{m}{k} \binom{m}{k'}  \p\Big( r( \{n^{1/d}X_i, Z_1,\ldots, Z_k \} ) \in  (a,b], \nonumber \\
		&\qquad\qquad\qquad\qquad\qquad \qquad\qquad X_i \neq Y_i \text{ or } Z_j \neq \wt Z_j \text{ for a } j\in \{1,\ldots,k\} \Big)\nonumber \\
		&\hspace*{2cm}\le C  |b-a| \epsilon \sum_{m=1}^{n} \p(N=m) e^{\wt c m} m^q \label{E:FunctionEulerApproximation10}
\end{align}
for certain constants $q, \wt c,C < \infty$.
The conclusion now follows from a Poissonization argument as the probability of $n^{1/d} X_j$ hitting  $ B( n^{1/d} X_i,\delta) $ is $\alpha n^{-1}$ for a constant $\alpha \in [0 \vee (\inf \kappa - \rho), \sup \kappa + \rho]$. This shows that \eqref{boundE:FunctionEulerApproximation5} (and thus also \eqref{E:FunctionEulerApproximation5}), is at most $C \ T \ J^{-1} \ \epsilon$ for a certain $C < \infty$ which is independent of $i\in [n]$, $n$ and independent* of $\nu$.

It is now straightforward to see that the term in \eqref{E:FunctionEulerApproximation6} is bounded above in a similar fashion by $C n ( T \ J^{-1} \ \epsilon )^2$. We omit the details.
\end{proof}

\subsection{Asymptotic normality}
In order to verify the asymptotic normality, we first show the strong stabilizing property of the EC.
\begin{proposition}[Strong stabilization]\label{P:StrongStabilizingProperty}
Let $t \in [0,\infty)$. Consider the {\v C}ech or the Vietoris-Rips complex $\cK_t( P )$ obtained from a locally finite point cloud $P$. 

Write
$	\Delta(t, P ) = \chi( \cK_t( \{0\} \cup P ) ) - \chi( \cK_t( P ) )$.
Define the radius of stabilization $S\coloneqq 2t$. There is a random variable $\Delta_\infty(t,P)\in\R$ such that 
$$\Delta(t, (P \cap B(0,S) ) \cup A ) = 	\Delta_\infty(t,P)$$ for all finite $A \subseteq \R^d \setminus B(0,S)$.

In particular, the equalities in \eqref{D:DeltaInf} and \eqref{D:DeltaInf2} are true.
\end{proposition}
\begin{proof}
Consider the difference
\begin{align*}
		\Delta(t, P ) &= \chi( \cK_t( \{0\} \cup P ) ) - \chi( \cK_t( P ) ) \\
		&= \sum_{k=0}^\infty (-1)^k \left\{ S_k( \cK_t( \{ 0\} \cup P) ) - S_k (\cK_t(P) ) \right\}
\end{align*}
which is determined by the points inside the $S$-neighborhood of 0, $B(0,S)$. Moreover, let $n$ be the number of points of $P$ in $B(0,S)$, i.e. $P \cap B(0,S) = \{ z_1,\ldots, z_n \}$ for generic points $z_1,\ldots,z_n$. Then $\Delta_\infty(t,P)$ equals
\[
		\Delta_\infty(t,P) = \sum_{k=0}^n (-1)^k \sum_{ (i_1,\ldots,i_k) } \1{ r(0,z_{i_1},\ldots,z_{i_k} ) \le t },
\]
where the second sum is taken over all $k$-tuples $\{i_1,\ldots,i_k\} \subseteq [n]$ such that $i_u \neq i_v$ for all pairs $(u,v), u\neq v$.

Finally, we prove the amendment. It is clear from the above that \eqref{D:DeltaInf} is true. Moreover, noting that a change of the configuration of Poisson points inside a box $Q(z)$ with edge length 1 only affects points inside a $(2t+\sqrt{d})$-neighborhood of $z$, shows \eqref{D:DeltaInf2}.
\end{proof}

\begin{proposition}[Positive variance]\label{P:Positivity}
Let $\alpha(t) = \mathbb{E}[ \Delta_\infty ( t ) ] $, and let $Z$ be a random variable with a bounded density $\kappa.$ Then, using the definition of $\gamma$ in \eqref{E:DefinitionGamma}
\[
			 \E{ \gamma( \kappa(Z)^{1/d} (t,t) )		} - \E{ \alpha( \kappa(Z)^{1/d} t)}^2 > 0,
\]
that is, the limit variances of $(\ol \chi_n(t))_n$ are positive for each $t>0$. 
\end{proposition}
\begin{proof}
We begin with the case of a uniform distribution $\kappa \equiv 1$. Then the limit variance in the binomial sampling scheme equals 
\begin{align}\label{E:Positivity1}
	\gamma(t,t) - \alpha(t)^2 = \E{ \E{ \fD_\infty(t,0)|\cF_0}^2 } - \E{ \Delta_\infty(t) }^2 > 0
\end{align}
for each $t >0$, the positivity follows from the fact that the distribution of  $ \Delta_\infty(t) $, $t>0$, is non degenerate and Penrose and Yukich \cite[Theorem 2.1]{penrose2001central}. (Observe that the second term in the last formula does not occur in the Poisson sampling scheme).

Let now $\kappa$ be a general density. We infer from Proposition~\ref{P:MVN} that the limit variance in the binomial sampling scheme takes the form
\begin{align*}
	&\int_{[0,1]^d}	\gamma( \kappa(x)^{1/d}(t,t) ) \kappa(x) \intd{x} - \Big(\int_{[0,1]^d}\alpha(\kappa(x)^{1/d} t) \kappa(x) \intd{x} \Big)^2 \\
	&\ge \int_{[0,1]^d} \big\{ \gamma( \kappa(x)^{1/d}(t,t) ) - ( \alpha(\kappa(x) ^{1/d} t ) )^2 \big\} \ \kappa(x) \intd{x},
\end{align*}
which is positive for $t >0$ by \eqref{E:Positivity1}.  
\end{proof}

\begin{proof}[Proof of Theorem~\ref{T:NormalApproximation}]
Let $t\in (0,T]$ be arbitrary but fixed. The proof is divided in two parts. First we derive \eqref{E:NormalApproximation}, then \eqref{E:KolmogorovApproximationEuler0}.

\textit{Derivation of \eqref{E:NormalApproximation}.}  Recall the definitions of  $\wt\fD'_n(t,i)$ and $\wt\fD'^A_n(t,i)$  given early in Section~\ref{Section_TechnicalResults}. Let $\mu_{n,t}$ denote the distribution of $\ol\chi_{\nu,n}(t) /\V(\ol\chi_{\nu,n}(t))^{1/2}$, and define 
\[
		S = \frac{1}{2} \sum_{ A \subsetneq [{b_n^\prime}]  } \sum_{ i \notin A } \frac{ \wt\fD'_n(t,i)  \wt\fD'^A_n(t,i) }{ \binom{{b_n^\prime}}{|A|} ({b_n^\prime} - |A|)  }
\quad \text{ and } \quad
		S^| = \frac{1}{2} \sum_{ A \subsetneq [{b_n^\prime}]  } \sum_{ i \notin A} \frac{ \wt\fD'_n(t,i) \ |\wt\fD'^A_n(t,i)| }{ \binom{{b_n^\prime}}{|A|} ({b_n^\prime} - |A|)  }.
\]
Then, using \cite[Theorem 2.2.]{chatterjee2008new} in case of the Kantorovich-Wasserstein distance,  and \cite[Theorem 4.2]{lachieze2017new} for the Kolmogorov distance, respectively, we obtain the following two inequalities: 
\begin{align}\label{E:NormalApproximation1}
	&d_W( \mu_{n,t}, \cN_{0,1} ) \nonumber \\
	&\le \frac{1}{\sigma^2} \V(S)^{1/2} + \frac{1}{2\sigma^3} \sum_{i=1}^{{b_n^\prime}} \E{ | \wt\fD'_n(t,i) |^3 }, \\
	& d_K( \mu_{n,t}, \cN_{0,1} ) \nonumber \\
	\begin{split}\label{E:NormalApproximation2}
	 &\le \frac{1}{\sigma^2} \V(S)^{1/2} + \frac{1}{\sigma^2} \V(S^|)^{1/2} \\
	&\quad + \frac{1}{4 \sigma^3} \sum_{i=1}^{{b_n^\prime}} \E{ | \wt\fD'_n(t,i) |^6 }^{1/2} + \frac{\sqrt{2\pi} }{16 \sigma^3} \sum_{i=1}^{{b_n^\prime}} \E{ | \wt\fD'_n(t,i) |^3 },
	\end{split}
\end{align}
where $\sigma^2 = \V( \chi_{\nu,n}( \cK_t) ).$ It thus remains to bound the terms on the right-hand sides.

Using Fatou's lemma,  $\liminf_{n\to \infty} \V( \ol \chi_{\nu,n} (t) ) \ge c^*$ for some positive $c^*\in\R$ for all $\nu \in B_\infty(\kappa,\rho)$. Here $\rho$ must be sufficiently small, see \eqref{E:NormalApproximation3} below. Moreover, we will see that $c^*$ depends on $t$. Indeed, $\liminf_{n\to \infty} \V( \ol \chi_{\kappa,n} (t) )^{1/2} \ge c_1 > 0$ by Proposition~\ref{P:Positivity}, where $c_1$ depends on $t$. Furthermore, using the result of Theorem~\ref{T:ApproximationEuler}, there is a constant $c_2$ such that $\mathbb{E}[ (\ol \chi_{\kappa,n} (t) - \ol \chi_{\nu,n} (t) )^2 ]^{1/2} \le c_2 \|\kappa-\nu\|_\infty^{1/2}$ whenever $\|\kappa-\nu\|_\infty \le \wt \rho$, for some $\wt \rho>0$ which is fixed. (Note that $c_2$ can be chosen uniformly in $t\in[0,T]$.) Hence, 
\begin{align}
			& \liminf_{n\to \infty} \Big\{ \inf_{\nu \in B_\infty(\kappa,\rho)} \sqrt{\V( \ol \chi_{\nu,n} (t) )} \Big\} \nonumber \\
			&= \liminf_{n\to \infty} \Big\{ \inf_{\nu \in B_\infty(\kappa,\rho) }\sqrt{\E{ \ol \chi^2_{\nu,n} (t) } } \Big\} \nonumber \\
			&\ge \liminf_{n\to \infty} \sqrt{\E{ \ol \chi^2_{\kappa,n} (t) }}  - \Big\{ \inf_{\nu \in B_\infty(\kappa,\rho) }\sup_{n\in\N} \sqrt{\E{ (\ol \chi_{\kappa,n} (t) - \ol \chi_{\nu,n} (t) )^2 }} \Big\} \nonumber \\
			& \ge  c_1 - c_2 \rho^{1/2} \label{E:NormalApproximation3}
\end{align}
which is positive if $\rho$ is sufficiently small. This implies, $\sigma^2 \ge c^* n$ for all but finitely many $n$ and uniformly in $\nu \in B_\infty(\kappa, \rho)$ for some  $c^* > 0$, which depends on $t$.

Furthermore,  Lemma~\ref{L:BoundedMomentsCondition} says that, for each $p\in\N$, $\mathbb{E}[ | \wt\fD'_n(t,i) |^p ]$ is bounded above uniformly over all $n$, $i\in [{b_n^\prime}] $ and $\nu \in B_\infty(\kappa,\wt \rho)$. Hence given $t$, the second term in \eqref{E:NormalApproximation1} and the third and fourth term in \eqref{E:NormalApproximation2} are of order ${b_n^\prime} / n^{3/2}$ which is of order $n^{-1/2}$. It remains to obtain bounds for $\V(S)$ and $\V(S')$. We only study $\V(S)$ in detail. The calculations will show that $\V(S')$ admits a very similar upper bound.
\begin{align}\begin{split}\label{E:CovarianceT}
		\V(S) &= \frac{1}{4} \sum_{i\in [{b_n^\prime}]}  \sum_{i ' \in [{b_n^\prime}]} \sum_{ A \subsetneq [{b_n^\prime}]: A \not\ni i  } \sum_{ A' \subsetneq [{b_n^\prime}]: A'\not\ni i'  } \\
		&\qquad\qquad\qquad  \frac{ \cov( \wt\fD'_n(t,i) \wt\fD'^{A}_n(t,i) , \wt\fD'_n(t,i') \wt\fD'^{A'}_n(t,i') ) }{  \binom{{b_n^\prime}}{|A|} ({b_n^\prime} - |A|) \binom{{b_n^\prime}}{|A'|} ({b_n^\prime} - |A'|) },
\end{split}\end{align}
where all covariances are uniformly bounded by Lemma~\ref{L:BoundedMomentsCondition}.

First consider the Poisson case. Clearly, $\wt\fD'_n(t,i), \wt\fD'^A_n(t,i)$ both only involve Poisson points in a $\delta$-neighborhood of $z_{n,i}$ and $\delta$ does not depend on $t, A, z_{n,i}$, see also Proposition 4.5. Consequently, exploiting the independence property of the Poisson process, for a given $i$, the number of indices $i'$ such that the covariances in \eqref{E:CovarianceT} are non zero does not depend on $n$ and is bounded above by some constant. Moreover, a combinatorial argument shows that
\[
	\sum_{A\subsetneq [{b_n^\prime}], A \not\ni i } \frac{1}{ \binom{ {b_n^\prime}}{|A|} ({b_n^\prime} - |A| ) } \equiv 1. 
\]
Hence, $\V(S)$ is of order ${b_n^\prime}$. This completes the calculations in the Poisson case.

Now we bound (\ref{E:CovarianceT}) in the binomial case. If $i= i'$, then the remaining double sum is bounded above by a constant. To see this we use Cauchy-Schwarz and the boundedness of the variances together with the previous displayed formula. If $i\neq i'$, then we need to consider the covariances between simplices in a $\delta$-neighborhood around $X_i$, $X'_i$ and $X_{i'}$, $X'_{i'}$. The covariance is only non zero if the distance between $\{X_i, X'_i\}$ and $\{X_{i'}, X'_{i'}\}$ is at most $\delta$. This happens with a probability proportional to $n^{-1}$. This shows once more that $\V(S)$ is of order $n$ and establishes the claim in the case of a binomial sampling scheme. This demonstrates \eqref{E:NormalApproximation}.

\textit{Derivation of \eqref{E:KolmogorovApproximationEuler0}} is now straightforward. Letting $\cN_{0,1}(\cdot)$ denote the law of the standard normal distribution, write
\begin{align}
		d_K( \ol \chi_{\nu,n} (t), \ol \chi_{\kappa, n} (t) ) &\le d_K\Big( \ol \chi_{\nu,n} (t) / \sqrt{\V( \ol \chi_{\nu,n} (t) )}, \cN_{0,1} \Big) \label{E:KolmogorovApproximationEuler1} \\
		\begin{split}\label{E:KolmogorovApproximationEuler2} 
		&\quad + d_K\Big(  \cN_{0,1} \big( \cdot/ \sqrt{\V( \ol \chi_{\nu,n} (t) )} \big) , \\
		&\qquad\qquad\qquad \cN_{0,1} \big( \cdot/ \sqrt{\V( \ol \chi_{\kappa,n} (t) )} \big) \Big)
		\end{split}   \\
		&\quad + d_K \Big( \cN_{0,1},  \ol \chi_{\kappa,n} (t) / \sqrt{\V( \ol \chi_{\kappa,n} (t) )}  \label{E:KolmogorovApproximationEuler3}\Big).
\end{align}
If  $\nu \in B_\infty(\kappa,\rho)$, where $\rho$ is selected as above, the previous results regarding the normal approximation show that \eqref{E:KolmogorovApproximationEuler1} and \eqref{E:KolmogorovApproximationEuler3} attain the rate $C n^{-1/2}$, which is uniform in $\nu \in B_\infty(\kappa,\rho)$ for a given $t\in [0,T]$. 

Regarding \eqref{E:KolmogorovApproximationEuler2}, we use the following continuity result of the standard normal distribution
\[
		\sup_{u\in\R} | \cN_{0,1}( a + bu ) - \cN_{0,1}( u ) | \le |a| + |b| \vee (|b|^{-1}) -1.
\]
Thus, we are left to study the quotients
\begin{align}\label{E:KolmogorovApproximationEuler4}
	\left|		\sqrt{\frac{ \V( \ol \chi_{\nu,n} (t) ) }{ \V( \ol \chi_{\kappa,n} (t) ) }} - 1 \right | \text{ and }  \left|		\sqrt{\frac{ \V( \ol \chi_{\kappa,n} (t) ) }{ \V( \ol \chi_{\nu,n}  (t) ) }} - 1 \right | .
\end{align}
Given two centered square integrable random variables $Z_1,Z_2$, we have by the reverse triangle inequality
\begin{align*}
	|\E{ Z_1^2}^{1/2} - \E{ Z_2^2}^{1/2}| &\le \E{ (Z_1 - Z_2)^2 }^{1/2} .
\end{align*}
Thus, recalling Theorem~\ref{T:ApproximationEuler},
$		 | \V( \ol \chi_{\nu,n} (t) )^{1/2} - \V( \ol \chi_{\kappa,n} (t) )^{1/2}  | 		\le \sqrt{C_{0,\kappa}} \| \nu - \kappa \|_\infty^{1/2}$ 
uniformly over $\nu \in B_\infty(\kappa,\wt \rho)$. Using \eqref{E:NormalApproximation3}, we see that both terms in \eqref{E:KolmogorovApproximationEuler4} are of order $ \| \nu - \kappa \|_\infty^{1/2}$ uniformly over $\nu \in B_\infty(\kappa,\rho)$ for some $\rho >0$ small enough given $t$. This completes the proof.
\end{proof}

\begin{proof}[Proof of Theorem~\ref{T:FunctionalCLT}]
Convergence of the finite dimensional distributions is shown in Proposition~\ref{P:MVN}. In the subsequent propositions, we show stochastic equicontinuity (Proposition~\ref{P:Tightness}) and the H{\"o}lder continuity of the sample paths (Proposition~\ref{P:ContModi}).
\end{proof}

\begin{proposition}[Tightness]\label{P:Tightness}
Consider a Poisson or a binomial sampling scheme and assume the conditions of Theorem~\ref{T:FunctionalCLT}. The family of probability measures $\{ \cL(	(\ol \chi_n(t): t\in [0,T])	) : n\in\N \}$ is tight in the Skorohod $J_1$-topology.
\end{proposition}
\begin{proof}[Proof of Proposition~\ref{P:Tightness}]
The tightness will be established with \cite[Theorem 3]{bickel1971convergence} (which is an extension of \cite[Theorem 15.6]{billingsley1968convergence}) and the remark following this theorem. Let $I_1$ be the even integers in $\N_0$ and let $I_2$ be the odd integers in $\N$. Define the two c{\` a}dl{\` a}g processes
$$
	[0,T]\ni t \mapsto	\Sigma_{i,n}(t) =  \sum_{q\in I_i }  S_q(\cK_{t,n})
$$
for $n\in\N$ and $i\in\{1,2\}$. Then
$$
 | \ol\chi_n(t) | \le \frac{1}{\sqrt{n}}\sum_{i=1,2} \Big| \Sigma_{i,n}(t) -  \E{ \Sigma_{i,n}(t) } \Big |.
$$
and
$$
 | \ol \chi_n(s) - \ol \chi_n(t)  | \le \frac {1}{\sqrt{n}}\sum_{i=1,2} \Big| \Sigma_{i,n}(s)   - \Sigma_{i,n}(t) -  \E{\Sigma_{i,n}(s)  - \Sigma_{i,n}(t) } \Big |.
$$
Hence, $( n^{-1/2} (\chi_n(\cdot) - \E{\chi_n(\cdot) }) )_{n\ge 1}$ is tight if $( n^{-1/2} (\Sigma_{i,n}(\cdot) - \E{\Sigma_{i,n}(\cdot) }) )_{n\ge 1}$ is for both $i=1$ and $i=2$ (we refer to \cite{billingsley1968convergence} for classical tightness results in terms of the modulus of continuity).

Consequently, the rest of the proof is dedicated to verify the tightness of the sequence $( n^{-1/2} (\Sigma_{i,n}(\cdot) - \E{\Sigma_{i,n}(\cdot) }) )_{n\ge 1}$ for $i\in\{1,2\}$. We proceed in two parts. Using the amendment of the main theorem in \cite{bickel1971convergence}, we first show that it is sufficient to compute the modulus of continuity of the processes $( n^{-1/2} (\Sigma_{i,n}(\cdot) - \E{\Sigma_{i,n}(\cdot) }) )_{n\ge 1}$ on a reduced grid $\Gamma_n$ of $(n+1)$ equidistant points on $[0,T]$. In the second part, we verify the moment condition. The calculations are essentially the same for the binomial and the Poisson sampling scheme. 

\textit{Part 1.} 
We show that for $i\in\{1,2\}$
\begin{align}\begin{split}\label{E:ModulusContinuity1}
		 &\sup_{ \substack{s,t\in [0,T] } } \frac{1}{\sqrt{n}}\Big| \Sigma_{i,n}(s)   - \Sigma_{i,n}(t) -  \E{\Sigma_{i,n}(s)  - \Sigma_{i,n}(t) } \Big | \\
		& \le  \sup_{ \substack{s,t\in \Gamma_n } } \frac{1}{\sqrt{n}}\Big|  \Sigma_{i,n}(s)   - \Sigma_{i,n}(t) -  \E{\Sigma_{i,n}(s)  - \Sigma_{i,n}(t) } \Big |+ \frac{C}{\sqrt{n}},
\end{split}\end{align}
for some constant $C\in\R_+$ not depending on $n$. To see this, let $0 \le s \le t \le T$ and $\underline s, \ol s, \underline t, \ol t \in \Gamma_n$ with $\underline s \le s \le \ol s,\,\underline t \le t \le \ol t$ and $|\ol t - \underline t| \le T/n, |\ol s - \underline s| \le T/n$. Then, due to  monotonicity of $s \mapsto \Sigma_{i,n}(s)$ ,
\begin{align*}
\big(\Sigma_{i,n}(t) -& \E{\Sigma_{i,n}(t)} \big) - \big(\Sigma_{i,n}(s) - \E{\Sigma_{i,n}(s)} \big) \\
& \le \big(\Sigma_{i,n}(\ol t) - \E{\Sigma_{i,n}(\ol t)} \big) - \big(\Sigma_{i,n}(\underline s) - \E{\Sigma_{i,n}(\underline s)} \big) \\
&\quad + \big( \E{\Sigma_{i,n}(\ol t)} - \E{\Sigma_{i,n}(\underline t)}\big) - \big( \E{\Sigma_{i,n}(\ol s)} - \E{\Sigma_{i,n}(\underline s)}\big).
\end{align*}
This gives together with a similar lower bound
\begin{align*}
		 &\sup_{ \substack{s,t\in [0,T] } } \frac{1}{\sqrt{n}}\Big| \Sigma_{i,n}(s)   - \Sigma_{i,n}(t) -  \E{\Sigma_{i,n}(s)  - \Sigma_{i,n}(t) } \Big | \\
		&\hspace*{1cm} \le  \sup_{ \substack{s,t\in \Gamma_n } } \frac{1}{\sqrt{n}}\Big|  \Sigma_{i,n}(s)   - \Sigma_{i,n}(t) -  \E{\Sigma_{i,n}(s)  - \Sigma_{i,n}(t) } \Big | \\
		&\hspace*{1cm} \quad+ 2 \sup_{s,t \in \Gamma_n \atop |s - t| \le T/n} \big( \E{\Sigma_{i,n}(t)} - \E{\Sigma_{i,n}(s)} \big).
\end{align*}
It remains to show that the last summand can be bounded by a constant $C > 0$. To this end, let $0\le s \le t \le T$ with a distance of at most $T/n$. Then 
\begin{align}
		&\E{\Sigma_{i,n}(t)} - \E{\Sigma_{i,n}(s)} \nonumber \\
		\begin{split}\label{E:ModulusContinuity2}
		&= \sum_{q \in I_i }  \E{ S_q( \cK_{t,n}) -  S_q( \cK_{s,n})  } \\
		& \le \sum_{i=1}^{{b_n^\prime}} \sum_{q\in \N_0} \E{ \sum_{\substack{ \sigma\in \cK_{t,n}, \dim \sigma = q}} \1{ \sigma\cap Q(z_{n,i}) , r(\sigma)\in (s,t] } }.
\end{split}\end{align}
Let $M=M_{n,i}$ be the number of points of $\cP(n)$ resp.\ $n^{1/d}\mX_n$ lying in $Q(z_{n,i})^{(\delta)}$ with $\delta = 2T$. Using Lemmas~\ref{L:ContinuityVR} and \ref{L:ContinuityCech}, we have conditional on $M$ 
\begin{align*}
		&\E{ \sum_{\substack{ \sigma\in \cK_{t,n}, \dim \sigma = q}} \1{ \sigma\cap Q(z_{n,i}) , r(\sigma)\in (s,t] } } \\
		&\le \1{q\le M-1} \binom{M}{q+1} C_q |t -s|
\end{align*}
for a constant $C_q$ which is uniform in $n$, $i\in[{b_n^\prime}]$ and $s,t$, moreover $C_q$ is independent* of $\nu$ and satisfies $C_q \le c e^{c q}$. In the Poisson sampling scheme, one uses that $M$ has a Poisson distribution with a parameter $\lambda\in\R_+$ that is uniformly bounded above in $n$ and $i\in[{b_n^\prime}]$. Then using Lemma~\ref{L:MomentsEuler}, we find that the right-hand side of \eqref{E:ModulusContinuity2} is of order
\begin{align*}
		&\sum_{i=1}^{{b_n^\prime}} \sum_{m=0}^{\infty} \p( M_{n,i} = m ) \sum_{q=0}^{m-1} \binom{M}{q+1} e^{c q } |s-t| \\
		&\le C \sum_{m=0}^{\infty} e^{-\lambda} \frac{ \lambda^m }{m!} \sum_{q=0}^{m-1} 2^q e^{cq} m^c \ {b_n^\prime} |s-t| \\
		&\le C \sum_{m=0}^{\infty} e^{-\lambda} \frac{ \lambda^m }{m!} \ \Big( 2^m e^{cm} m^c \Big) \ {b_n^\prime} \ |s-t| \\
		&\le C {b_n^\prime} |s-t|.
\end{align*}
Since $|s-t| \le T/n$ and ${b_n^\prime} / n \to 1$, we see that, in the Poisson sampling scheme, $\E{ |\Sigma_{i,n}(s)  - \Sigma_{i,n}(t)| }$ is bounded above by a constant uniformly in $s,t$ with $|s-t| \le T/n$ and $n\in\N$.

It is a standard routine to verify the same statement for the binomial sampling scheme, using the fact that in this case $M_{n,i}$ tends to a Poisson distribution. This concludes the proof of \eqref{E:ModulusContinuity1}.


\textit{Part 2.} We verify the moment condition from Bickel and Wichura (1971) for the maps 
$$
	\Gamma_n \ni t \mapsto n^{-1/2} \sum_{q\in I } S_q(\cK_{t,n}) - \E{ S_q(\cK_{t,n}) },
$$
where $I\in \{ I_1, I_2 \}$. We write $\Delta_{q,n}(s,t)$ for $S_q(\cK_{t,n})  -  S_q(\cK_{s,n})$. Let $0 \le r \le s \le t \le T$ be elements in $\Gamma_n$. We show
\begin{align}\begin{split}\label{E:MomentCondition1}	
		 &n^{-2} \ \mathbb{E} \Big[\Big( \sum_{q\in I } \Delta_{q,n}(r,s) - \E{ \Delta_{q,n}(r,s) } \Big)^2  \\
		 &\qquad\qquad\qquad \Big( \sum_{q\in I } \Delta_{q,n}(s,t) - \E{ \Delta_{q,n}(s,t) } \Big)^2   \Big] \\
		 &\le C |s-r| |t-s|
\end{split}\end{align}
for a constant $C$ not depending on $n$ and $r,s,t \in \Gamma_n$, $r\le s \le t$.

First, for any $s \le t$, we rewrite $\Delta_{q,n}(s,t) - \E{ \Delta_{q,n}(s,t) }$ as a martingale difference sequence
\[
		\Delta_{q,n}(s,t) - \E{ \Delta_{q,n}(s,t) } = \sum_{i=1}^{{b_n^\prime}} \E{ \Delta_{q,n}(s,t) - \wt \Delta_{q,n,i}(s,t) | \cG_{n,i} },
\]
where the filtration $(\cG_{n,i}: i=1,\ldots,{b_n^\prime} )$ is introduced at the beginning of Section~\ref{Section_TechnicalResults} and where $\wt \Delta_{q,n,i}(s,t) = S_q( \cK'_{t,n,i} ) - S_q( \cK'_{s,n,i} )$. With these abrevations the left-hand side of \eqref{E:MomentCondition1} equals
\begin{align}\begin{split}\label{E:MomentCondition2}
		&\sum_{i,j,k,\ell \in {\rm DM}} \sum_{q_1,\ldots,q_4 \in I } \  n^{-2}  \mathbb{E} \Big[  \E{ \Delta_{q_1,n}(s,t) - \wt \Delta_{q_1,n,i}(s,t) \ | \ \cG_{n,i} } \\
		&\quad\qquad\qquad\qquad\qquad\qquad  \E{ \Delta_{q_2,n}(s,t) - \wt \Delta_{q_2,n,j}(s,t) ) \ | \ \cG_{n,j} }\\
		&\quad\qquad\qquad\qquad\qquad\qquad \E{ \Delta_{q_3,n}(r,s) - \wt \Delta_{q_3,n,k}(r,s) \ | \ \cG_{n,k} } \\
		&\quad\qquad\qquad\qquad\qquad\qquad   \E{ \Delta_{q_4,n}(r,s) - \wt \Delta_{q_4,n,\ell}(r,s) \ | \ \cG_{n,\ell} } \Big], 
\end{split}\end{align}	
where the outer summation is extended over DM, the set of quadruples with "(at least a) double maximum" meaning the set  of elements in $[b'_n]^4$ for which the greatest index appears at least twice. All other index combinations have expectation 0 and thus do not enter the sum. We divide the remainder of the proof in two steps.

\textit{Step 1.} We show that we can reduce the sum over ${\rm DM}$ to a few smaller sums, each of which is extended over only $O((b'_n)^2)$ indices, and one sum over three indices (of the order $O(b'_n)^3$ indices), but with an extra factor of $n^{-1}$. This then enables us to derive the desired bounds, which is indicated in Step 2 below. 

Due to the symmetry of the situation, it is enough to study the subcases of (\ref{E:MomentCondition2}) with $i < j < \ell = k$ and $i < k < \ell = j$. In these cases the summands in \eqref{E:MomentCondition2} have the structure $\mathbb{E}[ F( W, V ) G(W, V, R ) ]$, where $W$ are the observations with index less than $i$, $V$ are the observations with index $i$ and $R$ are the observations larger than $i$. To be more precise, the first conditional expectation in (\ref{E:MomentCondition2}) is $F(W,V)$ with $W = \cP(n) \cap (Q(z_{n,1} \cup \cdots \cup Q(z_{n,i-1})$, or  $W = (X_1,\ldots,X_{i-1})$, respectively, and $V = \cP(n) \cap Q(z_{n,i})$, or  $V = X_i$, respectively. The last three conditional expectations in (\ref{E:MomentCondition2}) are not only functions of $V$ and $W$, but also of  $R = \cP(n) \cap (Q(z_{n,i+1}) \cup \cdots \cup Q(z_{n,\ell})$,  or  $R = (X_{i+1},\ldots,X_{\ell})$, respectively. 

Clearly, $W,V,R$ are independent. Hence, if we omit the variable $V$ as input in the second factor, we see by using an independence argument that $\mathbb{E}[ F(W,V) G(W,R) ] = 0$ because $\mathbb{E}[F(w,V)] = 0$ for almost every realization $w=W$. It thus suffices to study the difference $ F(W,V) (G(W,V,R) - G(W,R) )$. To make this more clear, let
\begin{align*}
		\Delta'_{q,n,i}(s,t) &= 
		\begin{cases}
			S_q( \cK_t (\cP(n) \setminus Q(z_{n,i}) ) ) , &\qquad \quad \text{Poisson sampling,} \\
			\qquad\quad - S_q( \cK_s (\cP(n) \setminus Q(z_{n,i}) ) ) \\
			S_q( \cK_t ( n^{1/d} [ \mX_n \setminus \{X_i \} ] ) ), & \qquad \quad \text{binomial sampling,} \\
			\qquad\quad  - S_q( \cK_s ( n^{1/d} [ \mX_n \setminus \{X_i \} ] ) )
		\end{cases} \\
	\wt  \Delta'_{q,n,j,i}(s,t) ) &=
	\begin{cases}  S_q( \cK_t( \{ [\cP(n) \setminus Q(z_{n,j})] \\
	\qquad \cup [ \cP'(n) \cap Q(z_{n,j})]  \} \setminus Q(z_{n,i}) ) ) \\
       \quad - S_q( \cK_s( \{ \ [\cP(n) \setminus Q(z_{n,j})] \\
       \qquad \cup [ \cP'(n) \cap Q(z_{n,j})] \ \} \setminus Q(z_{n,i}) ) ), & \ \text{Poisson sampling,} \\
S_q( \cK_t( \{  n^{1/d} [\{X'_j\} \cup (\mX_n \setminus \{X_i, X_j\}) ] )   ) \\
\quad - S_q( \cK_s( \{  n^{1/d} [\{X'_j\} \\
\qquad \cup (\mX_n \setminus \{X_i, X_j\}) ] )   ),  & \ \text{binomial sampling}.
\end{cases}
 \end{align*}
Let us consider the case $i < j < k = \ell$, and set  
$$
	D(q_1,i) = \E{ \Delta_{q_1,n}(s,t) - \wt \Delta_{q_1,n,i}(s,t) \ | \ \cG_{n,i} }
$$
 and $D'(q_1,j,i) = \E{ \Delta'_{q_1,n,i}(s,t) - \wt \Delta'_{q_1,n,j,i}(s,t) \ | \ \cG_{n,i} }.$ Similarly we define $D(q_2,i)$ and $D'(q_2,j,i)$, while  $D(q_3,\ell),$ and $D'(q_3,\ell,i)$ are defined with $(s,t)$ replaced by $(r,s)$, and the same holds for $D(q_4,\ell),$ and $D'(q_4,\ell,i)$.

Using this notation, the expression $F(W,V) (G(W,V,R) - G(W,R) )$ is of the form
 \begin{align*}
 {\mathbb E}\Big[ D(q_1,i) \big[D(q_2,j) D(q_3,\ell) D(q_4,\ell) - D'(q_2,j,i)D(q_3,\ell,i) D(q_4,\ell,i)\big]\Big]
 \end{align*}
and with $D(q_3,q_4,\ell) = D(q_3,\ell) D(q_4,\ell)$ and $D'(q_3,q_4,\ell,i) = D'(q_3,\ell,i) D'(q_4,\ell,i)$, straightforward calculations show that this can be written as
%
%
%
%
\begin{align*}
&{\mathbb E}\Big[ D(q_1,i)  \big[D(q_2,j) - D'(q_2,j,i)\big] D(q_3,q_4,\ell)\Big]\\
& \hspace*{1cm}+ {\mathbb E}\Big[ D(q_1,i)  \big[D(q_3,\ell) - D'(q_3,\ell,i)\big] D(q_4,\ell)D'(q_2,j,i)\Big]\\
&\hspace*{2cm} + {\mathbb E}\Big[ D(q_1,i) \big[D(q_4,\ell) - D'(q_4,\ell,i)\big] D'(q_3,\ell,i)D'(q_2,j,i)\Big].
 \end{align*}
 Each of the last three summands involves a factor of the form 
 \begin{align}
 \begin{split}\label{E:MomentCondition1b}
 &D(q,j) - D'(q,j,i) \\
 &= \E{ \Delta_{q,n}(s,t) - \wt \Delta_{q,n,i}(s,t)  - \Delta'_{q,n,i}(s,t) - \wt \Delta'_{q,n,j,i}(s,t)\ | \ \cG_{n,i}} 
\end{split} \end{align}
 for some $q$ and $j \ne i$ (note that $(r,s)$ can also take the role of $(s,t)$ here). We now study this difference inside this conditional expectation.
In the Poisson case, we write $ i  \simeq j$ for $Q(z_{n,i})^{(\delta)} \cap  Q(z_{n,j})^{(\delta)} \neq \emptyset$, and $i \not\simeq j$ otherwise. In the binomial sampling scheme, the notation $i \simeq j$ (resp. $\not\simeq$) simply means $i= j$ (resp. $\ne$). 

Using this notation, one finds that  in the Poisson sampling scheme \eqref{E:MomentCondition1b} is only non zero if $i \simeq j$ because both $ \Delta_{q,n}(s,t) - \wt \Delta_{q,n,j}(s,t) ) $ and $  \Delta'_{q,n,i}(s,t) - \wt \Delta'_{q,n,j,i}(s,t) )  $ only involve $q$-simplices with filtration times in $(s,t]$ which intersect with $Q(z_{n,j})$; and these are identical if $i\not\simeq j$. So the sum over three indices in \eqref{E:MomentCondition1} reduces to a sum over essentially two indices only; more precisely, the number of summands is of the order $O\big((b'_n)^2\big)$.

If $i \simeq j$ in the binomial sampling, then we obviously have a sum over two indices only, and \eqref{E:MomentCondition1b} consists of the $q$-simplices containing an element of $\{X_j,X'_j\}$. If $i\not\simeq j$, \eqref{E:MomentCondition1b} consists of $q$-simplices containing and element of both $\{X_j,X'_j\}$ and $\{ X_i, X'_i\}$, and crucially, this event is of order $n^{-1}$. So the latter sum, which is a sum over three indices, has an additional correction factor $n^{-1}$. With these insights and the techniques presented in the next step, it is straightforward to verify the claim for these subcases. 

Similar arguments hold in the case $i < k < j = \ell$. We omit further details in this step and continue with Step 2.

\textit{Step 2.} After having reduced the sums in Step 1, we now go back to (\ref{E:MomentCondition2}) and study this sum in the reduced settings. We verify the claim for the index combinations containing two pairs or one triple, so that the relevant index set has order $(b_n^\prime)^2$. Due to the symmetry of the situation, it is sufficient to study (a) $i=j$, $k = \ell$, (b) $i=k$, $j=\ell$ and (c) $i< j=k =\ell$. So, we have (at most) two indices only in each subcase (a) to (c); we write $i$ and $\ell$ for these.

The difference $\Delta_{q,n}(s,t) - \wt \Delta_{q,n,i}(s,t)$ consists only of simplices in a $\delta$-neighborhood of $Q(z_{n,i})$, or in a $\delta$-neighborhood of $X_i$ or $X'_i$, respectively, with a filtration time in $(s,t]$, i.e.,
\begin{align}
\begin{split}\label{E:MomentCondition3}
		&| \Delta_{q,n}(s,t) - \wt \Delta_{q,n,i}(s,t) | \\
		&\le
		\begin{cases}
			 \sum_{\sigma\in \cK_t \cup \cK'_{t,n,i}} \1{ r(\sigma) \in (s,t] , \sigma\cap Q(z_{n,i}) \neq \emptyset }, \text{ Poi. sampling}, \\
			\sum_{\sigma\in \cK_t \cup \cK'_{t,n,i}} \1{ r(\sigma) \in (s,t] , \sigma\cap \{X_i,X'_i\} \neq \emptyset },  \text{ bin. sampling}.
			\end{cases} 
			\end{split}
\end{align}
Furthermore, we can apply the following super-positioning principle of point processes. Clearly, we can compute the conditional expectation in \eqref{E:MomentCondition2} using five independent processes $\cP^{(0)}(n),\ldots, \cP^{(4)}(n)$, resp.\ $\mX^{(0)}_n,\ldots, \mX^{(4)}_n$. We use the process indexed by 0 for the outer expectation and the other four for each conditional expectation. Since this last upper bound is non decreasing in the number of points, we can use the joint process $\cP^*(n) = \cP^{(0)}(n) \cup \ldots \cup \cP^{(4)}(n)$, resp.\ $\mX^*_n = \mX^{(0)}_n \cup \ldots \cup \mX^{(4)}_n$ in the increments in \eqref{E:MomentCondition3} and obtain for the corresponding right-hand sides the following upper bounds
\begin{align}
\begin{split}\label{E:MomentCondition4}
		\begin{cases}
			 \sum_{\sigma\in \cK^*_t } \1{ r(\sigma) \in (s,t] , \sigma\cap Q(z_{n,i}) \neq \emptyset },  &\text{Pois. sampling}, \\
			\sum_{\sigma\in \cK^*_t} \1{ r(\sigma) \in (s,t] , \sigma\cap \{X^{(0)}_i,X^{(1)}_i,\ldots,,X^{(4)}_i\} \neq \emptyset }, &\text{bin. sampling},
			\end{cases} 
			\end{split}
\end{align}
where $\cK^*_t$ equals $\cK_t( \cP^*(n))$ resp.\ $\cK_t( \mX^*_n)$.

It is now a straightforward task, to calculate the relevant probabilities in the Poisson and binomial sampling scheme. We begin with the Poisson sampling scheme, where we use the spatial independence. For this let $\sigma_1$ (resp.\ $\sigma_2, \sigma_3,\sigma_4$) be a generic simplex intersecting with $Q(z_{n,i})$ (resp.\ $Q(z_{n,j}), Q(z_{n,k}), Q(z_{n,\ell})$). Then by Lemma~\ref{L:ContinuityVR} and Lemma~\ref{L:ContinuityCech}
\begin{align}
		&\p( r( \sigma_{1} ) \in (s,t], r( \sigma_{2} ) \in (s,t], r( \sigma_{3} ) \in (r,s], r( \sigma_{4} ) \in (r,s] , \nonumber \\
		&\qquad\qquad \sigma_1 \cap Q_{z_{n,i} } \neq \emptyset, \sigma_2 \cap Q_{z_{n,j} } \neq \emptyset, \sigma_3 \cap Q_{z_{n,k} } \neq \emptyset, \sigma_4 \cap Q_{z_{n,\ell} } \neq \emptyset   ) \nonumber \\
		&\le \p( r( \sigma_{1} ) \in (s,t], r( \sigma_{4} ) \in (r,s], \sigma_1 \cap Q_{z_{n,i} } \neq \emptyset, \sigma_4 \cap Q_{z_{n,\ell} } \neq \emptyset   )\nonumber  \\
	&\le C_{\dim \sigma_1}  C_{\dim \sigma_4}  | t- s | |s- r| \1{i \not\simeq \ell} + C_{\dim \sigma_1}  |t-s|  \1{ i \simeq \ell}. \label{E:MomentCondition5}
\end{align}
Note that this upper bound applies to each subcase (a), (b) and (c) because we only need to study the interplay between the random variables associated to $i$ and $\ell$.

Before completing the proof, we first provide similar bounds in the binomial sampling scheme. This time $\sigma_1$ is a generic simplex intersecting with the set $\{X^{(0)}_i,\ldots,X^{(4)}_i \} $, a similar notation is used for the indices $j,k,\ell$. Again, we can reduce the situation as follows
\begin{align}
		&\p( r( \sigma_{1} ) \in (s,t], r( \sigma_{2} ) \in (s,t], r( \sigma_{3} ) \in (r,s], r( \sigma_{4} ) \in (r,s], \sigma_1 \cap \{X^{(0)}_i,\ldots,X^{(4)}_i \} \neq \emptyset, \nonumber \\
		&\qquad \sigma_2 \cap  \{X^{(0)}_j,\ldots,X^{(4)}_j \} \neq \emptyset, \sigma_3 \cap  \{X^{(0)}_k,\ldots,X^{(4)}_k \} \neq \emptyset,   \sigma_4 \cap \{X^{(0)}_\ell,\ldots,X^{(4)}_\ell \} \neq \emptyset   ) \nonumber \\
		&\le \p( r( \sigma_{1} ) \in (s,t], r( \sigma_{4} ) \in (r,s], \sigma_1 \cap \{X^{(0)}_i,\ldots,X^{(4)}_i \} \neq \emptyset, \sigma_4 \cap \{X^{(0)}_\ell,\ldots,X^{(4)}_\ell \} \neq \emptyset   ) \nonumber  \\
		&\le C_{\dim \sigma_1} C_{\dim \sigma_4} |r-s| |s-t| \p\Big( d(  n^{1/d} \{X^{(0)}_i,\ldots,X^{(4)}_i \} ,  n^{1/d} \{X^{(0)}_\ell,\ldots,X^{(4)}_\ell \} ) \ge \delta \Big) \nonumber   \\
		&\quad +  C_{\dim \sigma_1}  |r-s| \p\Big( d(  n^{1/d} \{X^{(0)}_i,\ldots,X^{(4)}_i \} ,  n^{1/d} \{X^{(0)}_\ell,\ldots,X^{(4)}_\ell \} ) \le \delta \Big) \nonumber  \\
		&\le C_{\dim \sigma_1} C_{\dim \sigma_4} |s-r| |t-s| +  C_{\dim \sigma_1}  |s- r| \ c n^{-1}.  \label{E:MomentCondition6}
\end{align}
We can now complete \eqref{E:MomentCondition2} using the upper bounds from \eqref{E:MomentCondition3} and \eqref{E:MomentCondition4} for the martingale differences, and the upper bounds on the probabilities from \eqref{E:MomentCondition5}  and \eqref{E:MomentCondition6} for the Poisson and the binomial sampling scheme, respectively.

In the Poisson sampling scheme, we have to consider the quantity
\begin{align}\begin{split}\label{E:MomentCondition7}
	n^{-2} \sum_{i,\ell = 1}^{{b_n^\prime}} \sum_{q_1,\ldots,q_4 = 0}^{\infty} \mathbb{E} &\Big[  \sum_{ \substack{ \sigma_1 \in \cK^*_t, \\ \dim \sigma_1 = q_1 } }  \sum_{ \substack{ \sigma_2 \in \cK^*_t, \\ \dim \sigma_2 = q_2 } }  \sum_{ \substack{ \sigma_3 \in \cK^*_t, \\ \dim \sigma_3 = q_3 } }  \sum_{ \substack{ \sigma_4 \in \cK^*_t, \\ \dim \sigma_4 = q_4 } } \\
	&\qquad \1{ r(\sigma_1) \in (s,t], \sigma_1 \cap Q(z_{n,i}) \neq \emptyset }  \\
	&\qquad \1{ r(\sigma_2) \in (s,t], \sigma_2 \cap Q(z_{n,j}) \neq \emptyset } \\
	&\qquad \1{ r(\sigma_3) \in (r,s], \sigma_3 \cap Q(z_{n,k}) \neq \emptyset } \\
	&\qquad \1{ r(\sigma_4) \in (r,s], \sigma_4 \cap Q(z_{n,\ell}) \neq \emptyset }  \Big].
	\end{split}
\end{align}
In the binomial sampling scheme, we replace the "cube intersection conditions" in \eqref{E:MomentCondition7} with the second line in \eqref{E:MomentCondition5} containing the "point intersection conditions".

We begin with the Poisson sampling scheme. Denote by $p_{n,i}(m)$ (resp.\ $p_{n,\ell}(m)$) the probability that $Q(z_{n,i})^{(\delta)}$ (resp.\  $Q(z_{n,\ell})^{(\delta)}$) contains $m$ Poisson points of $\cP^*(n)$. We write $p_{n,0}(m)$ for the probability that $Q(z_{n,i})^{(\delta)} \cup Q(z_{n,\ell})^{(\delta)} $ contain $m$ Poisson points of $\cP^*(n)$. $p_{n,i}$, $p_{n,\ell}$ and $p_{n,0}$ follow a Poisson distribution with a Poisson parameter which can depend on the indices but which is uniformly bounded above.

We begin with the subcase (a), where $i=j$ and $k=\ell$. Then \eqref{E:MomentCondition7} amounts to
\begin{align*}
		& n^{-2} \sum_{i,\ell = 1}^{{b_n^\prime}} \1{ i \not \simeq \ell} \sum_{m_1,m_2 = 0}^\infty p_{n,i}(m_1) p_{n,\ell}(m_2) \\
		&\quad \sum_{q_1,q_2 = 0}^{m_1 - 1} \sum_{q_3,q_4 = 0}^{m_2-1} \binom{ m_1}{q_1+1} \binom{ m_1}{q_2+1}  \binom{ m_2}{q_3+1}  \binom{ m_2}{q_4+1} \\
		&\qquad\qquad\qquad\qquad\qquad\qquad\qquad\qquad \times C_{q_1} C_{q_4} |s-r| |t-s| \\
		&\quad + n^{-2} \sum_{i,\ell = 1}^{{b_n^\prime}} \1{ i  \simeq \ell} \sum_{m_0 = 0}^\infty p_{n,0}(m_0)  \\
		&\qquad \sum_{q_1,\ldots,q_4 = 0}^{m_0 - 1} \binom{ m_0}{q_1+1} \binom{ m_0}{q_2+1}  \binom{ m_0}{q_3+1}  \binom{ m_0}{q_4+1} C_{q_1} |t-s| \\
		&\le C ( |s-r| |t-s| + |t-s| n^{-1} ) \le 2C |s-r| |t-s|,
\end{align*}
where the last inequality follows because the interval length is bounded below by $n^{-1}$. The subcases (b) and (c) follow similarly, the only difference is that the binomial coefficients change somewhat. The conclusion is the same. So we find in all three cases that \eqref{E:MomentCondition7} is at most $C |s-r| |t-s|$.

In the binomial sampling scheme, we replace the Poisson distributions $p_{n,i}, p_{n,\ell}$ and $p_{n,0}$ by their binomial approximations conditional on the sets $\{X^{(0)}_i, \ldots, X^{(4)}_i \}$, $\{X^{(0)}_\ell, \ldots, X^{(4)}_\ell \}$ and $\{X^{(0)}_i, \ldots, X^{(4)}_i \} \cup \{X^{(0)}_\ell, \ldots, X^{(4)}_\ell \}$. Also, we replace the factor $n^{-1}$ by $|t-s|$ or $|s-r|$ because the interval length is bounded below by $n^{-1}$. With these preparations, it is a routine to verify that \eqref{E:MomentCondition7} is bounded above by $C|s-r| |t-s|$ for a $C\in\R_+$ independent of $n$. 
\end{proof}

\begin{proposition}\label{P:ContModi}
There is a continuous modification of $\fG$, whose sample paths are locally $\beta$-H{\"o}lder continuous for each $\beta\in (0,1/2)$.
\end{proposition}

\begin{proof}[Proof of Proposition~\ref{P:ContModi}]
First we prove the claim for the Poisson sampling scheme. The claim for the binomial sampling scheme is then an immediate consequence as shown below.

 Let $0\le s \le t\le T$. We write $\fG_\kappa$ for the Gaussian limit of the EC, when the underlying density is $\kappa$. We use $\fG$ for the Gaussian limit if $\kappa$ is the uniform distribution on $[0,1]^d$.

Since $\fG_\kappa(t) - \fG_\kappa(s)$ follows a normal distribution, we have for each $k\in\N$
\begin{align}\label{E:C1}
	\E{ (\fG_\kappa(t) - \fG_\kappa(s) )^{2k} } = \prod_{i=1}^k (2i-1) \E{ (\fG_\kappa (t) - \fG_\kappa(s))^2 }^k.
\end{align}
Using the representation of the covariance function, we obtain in the Poisson case
\begin{align}
	&\E{ (\fG_\kappa(t) - \fG_\kappa(s))^2 } \nonumber \\
	&= \int_{[0,1]^d} \Big\{ \gamma( \kappa(z)^{1/d} (t,t)) - 2 \gamma( \kappa(z)^{1/d} (t,s) ) + \gamma( \kappa(z)^{1/d} (s,s)) \Big\} \kappa(z) \intd{z} \nonumber \\
	&=\int_{[0,1]^d} \E{ (\fG(\kappa(z)^{1/d} t) - \fG(\kappa(z)^{1/d} s))^2 } \kappa(z) \intd{z}.\label{E:C2}
\end{align}
Consider the expectation in \eqref{E:C2}. Using the definition of $\gamma$, we have for $s, t \ge 0,$
\begin{align}
	\E{ (\fG( t) - \fG( s))^2 } &= \E{ \E{ \fD_\infty(t,0) - \fD_\infty(s,0) | \cF_0 }^2 }  \nonumber \\
	 &\le \E{ (\fD_\infty(t,0) - \fD_\infty(s,0) )^2 }  \label{E:C3}
\end{align}
Given a simplicial complex $\cK$ and the set $Q=Q_0$, we write $S_k(\cK;Q)$ for the number of $k$-simplices in $\cK$ with one vertex in $Q$. Given the upper bound $T$, there is an $R\ge 0$ such that for all $0\le t \le T$, the limit $\fD_\infty(t,0)$ admits the representation
\begin{align}\begin{split}\label{E:C4}
	\fD_\infty(t,0) &= \sum_{q=0}^\infty (-1)^q \big\{ S_q( \cK_t(\cP \cap B(0,R)) ; Q ) \\
	&\qquad - S_q( \cK_t( [ (\cP \cap B(0,R))\setminus Q] \cup [\cP' \cap Q]) ; Q ) \big\}
	\end{split}
\end{align}
(see proof of Proposition~\ref{P:StrongStabilizingProperty}).

We can use the representation in \eqref{E:C4} to obtain the following upper bound for \eqref{E:C3} (up to a universal multiplicative constant)
\begin{align}\begin{split}\label{E:C5}
	&\mathbb{E}\Big[ \Big( \sum_{q=0}^\infty  \big\{ S_{q}( \cK_t(\cP \cap B(0,R)) ; Q ) -  S_{q}( \cK_s(\cP \cap B(0,R)) ; Q ) \big\} \Big)^2 \Big] .
	\end{split}
\end{align}
We follow the calculations as in \eqref{E:MVN1} to see that the expectation in \eqref{E:C5} is at most $C |s-t|$ for a universal constant $C$, which only depends on $T$.

Combining the estimates from \eqref{E:C1} to \eqref{E:C5} yields that $\E{ (\fG_\kappa(t) - \fG_\kappa(s) )^{2k} } \le C_k |t-s|^k$ for all $0\le s,t\le T$ for a universal constant $C_k\in\R_+$ for all $k\in\N$. Hence, by the Kolmogorov-Chentsov continuity theorem there is a continuous modification of $\fG$, which is $\beta$-H{\"o}lder continuous with exponent $\beta\in \cup_{k\ge 1}(0,(k-1)/(2k) )$.

If the binomial sampling scheme is used, we have
\begin{align}
		\E{ (\fG_\kappa(t) - \fG_\kappa(s))^2 } 
		&=\int_{[0,1]^d} \E{ (\fG(\kappa(z)^{1/d} t) - \fG(\kappa(z)^{1/d} s))^2 } \kappa(z) \intd{z} \nonumber \\
		&\quad - \Big\{ \int_{[0,1]^d} \big( \alpha( \kappa(z)^{1/d} t) - \alpha( \kappa(z)^{1/d} s ) \big) \ \kappa(z)  \intd{z} \nonumber \label{E:C6}
 \Big\}^2 \\
 &\le \int_{[0,1]^d} \E{ (\fG(\kappa(z)^{1/d} t) - \fG(\kappa(z)^{1/d} s))^2 } \kappa(z) \intd{z}  . \nonumber
\end{align}
Consequently, the claim follows from the previous arguments.
\end{proof}

\subsection{Results on the bootstrap}

\begin{proof}[Proof of Theorem~\ref{T:PointwiseValidity}]
The estimate $\hat\kappa_n$ is uniformly consistent, i.e. there is a random integer $N_0$ such that $\| \hat\kappa_n - \kappa\|_\infty\le \rho$ for all $n\ge N_0$. So, we can apply \eqref{E:ApproximationEuler0} from Theorem~\ref{T:ApproximationEuler} and \eqref{E:KolmogorovApproximationEuler0} from Theorem~\ref{T:NormalApproximation} to obtain the desired result.
\end{proof}

\begin{proof}[Proof of Theorem~\ref{T:Validity}]
The assertion of Theorem~\ref{T:Validity} is an immediate application of Theorem~\ref{T:FunctionalWasserstein}.
\end{proof}


\appendix

\section{Multivariate asymptotic normality}\label{Appendix}
The proof is quite similar to the proof of Theorem~2.1 and 3.1 in \cite{penrose2001central}. However, we cannot immediately apply their theorem because it formally only applies to the density function $\kappa=1$. Moreover, the EC is not necessarily polynomially bounded. Straightforward calculations show, however, that it is exponentially bounded. Indeed, let $P$ be a finite point cloud, then
\begin{align*}
		\chi( \cK_t(P)) &=  \sum_{k=0}^{ \# P -1 }  (-1)^k  S_k( \cK_t ) \le \sum_{k=0}^{ \# P -1 } \binom{\# P}{k+1} = 
		2^{\# P} - 1\le \exp( \# P ).
\end{align*}
Furthermore, as we treat the multivariate case, we want to obtain an analytic expression of the covariance structure of the Gaussian process appearing in the limit. For this we have to carry out the entire proof. However, since the EC is closely related to persistent Betti numbers, we can use the ideas laid out in \cite{krebs2019betti} as a blueprint and so we will only sketch the main points here.

\begin{proposition}[Multivariate asymptotic normality]\label{P:MVN}
Let $m\in\N$, $a_1,\ldots,a_m \in \R$ and $t_1,\ldots,t_m \in [0,T]$ be arbitrary but fixed. Let $Z_n$ be either the Poisson process $\cP(n)$ or the binomial process $n^{1/d}\mX_n$. Then $	\sum_{u=1}^m a_u \ \ol \chi( \cK_{t_u} ( Z_n ) )$ tends to a normal distribution as $n\to\infty$ with mean zero. In the Poisson case, the covariance is determined by the limit
$$
			\lim_{n\to\infty} \cov( \ol \chi( \cK_{s} ( \cP(n) ) ), \ol \chi( \cK_{t} ( \cP(n) ) ) ) = \mathbb{E}[ \gamma( \kappa(Z)^{1/d} (s,t) ) ],
			$$
where $Z$ has density $\kappa$ and $\gamma(s,t) =  \E{ \E{ \fD_\infty(s,0) \ | \ \cF_0} \ \E{\fD_\infty(t,0) \ | \ \cF_0}	} $, and in the binomial case by
\begin{align*}
		&\lim_{n\to\infty} \cov( \ol \chi( \cK_{s} ( n^{1/d} \mX_n ) ), \ol \chi( \cK_{t} ( n^{1/d} \mX_n ) ) ) \\
		&= \mathbb{E}[ \gamma( \kappa(Z)^{1/d} (s,t) ) ] -  \mathbb{E}[\alpha(\kappa(Z)^{1/d} s)] \mathbb{E}[\alpha(\kappa(Z)^{1/d} t) ],
\end{align*}
where $\alpha(t) =  \mathbb{E}[ \Delta_\infty (t) ] $.
\end{proposition}

\begin{proof}[Proof of Proposition~\ref{P:MVN}] Define $H$ by $n^{-1/2} H (Z_n) = \sum_{u=1}^m a_u \ \ol \chi( \cK_{t_u} ( Z_n ) )$. First, we consider the

\textit{Poisson sampling scheme.} Recalling that $
  		\cP'^{\,z}(n) = (\cP(n)\setminus Q(z ) \cup ( \cP'(n) \cap Q(z ) )$
and $\cP'^{\,i}(n) = \cP'^{\,z_{n,i}}(n)$ for  $i \in [n]$,  write
 \begin{align*}
	n^{-1/2} (H(\cP(n)) - \E{ H(\cP(n))} ) &= n^{-1/2} \sum_{i=1}^{{b_n^\prime}} \E{ H(\cP(n)) - H( \cP'^{\,i}(n)) | \cG_{n,i} } \\
	& =: n^{-1/2} \sum_{i=1}^{{b_n^\prime}} D_{n,i}.
	\end{align*}
We verify the conditions of the central limit theorem given in \cite{mcleish1974dependent}:
\begin{itemize}\setlength\itemsep{0em}
	\item [(1)] $\sup_{n\in\N} \frac{1}{b_n^\prime} \E{ \max_{1\le i \le {b_n^\prime} } D_{n,i} ^{2} } < \infty.$
	\item  [(2)] $\frac{1}{\sqrt{b_n^\prime}} \max_{1\le i\le {b_n^\prime}} |D_{n,i}| \to 0 $ in probability as $n\to\infty$.
	\item  [(3)] $\frac{1}{b_n^\prime} \sum_{i=1}^{{b_n^\prime}} D_{n,i} ^2 \to \sigma^2$ in $L^1(P)$ for some $\sigma>0$ depending on $(a_1,\ldots,a_m)^\prime$ and $(t_1,\ldots,t_m)^\prime$.
\end{itemize}
 The positivity of $\sigma$ follows from Proposition~\ref{P:Positivity}. Conditions (1) and (2) follow from Lemma~\ref{L:BoundedMomentsCondition}. Indeed, consider (1), which is less than $\sup_{n\in\N}  \max_{1\le i \le {b_n^\prime} } \E{  D_{n,i}^2  }$.
Now $ \mathbb{E}[ D_{n,i}^2 ]$ is bounded above in terms of the single differences $\mathbb{E}[ | \chi(\cK_{t_u,n}) - \chi( \cK'_{t_u,n,i}) |^2 ]$, $1\le u\le m$, and these expressions are bounded above uniformly in $n$ and $i$ by Lemma~\ref{L:BoundedMomentsCondition}. Regarding the property (2), we use that
\[
		\E{ \Big( {b_n^\prime}^{-1/2} \max_{1\le i\le {b_n^\prime}} |D_{n,i}| \Big)^4 } \le  \sup_{n\in\N}  \ \Big({b_n^\prime}^{-1} \  \max_{1\le i \le {b_n^\prime} } \E{  D_{n,i}^4  } \Big),
\]
which tends to zero by Lemma~\ref{L:BoundedMomentsCondition}. This shows (2). Finally, we verify (3). Clearly, 
$$	\sum_{i=1} ^{{b_n^\prime}} D_{n,i}^2 = \sum_{u,v=1}^m a_u a_v  \sum_{i=1}^{{b_n^\prime}} D_{n,i}(t_u) D_{n,i} (t_v),$$ where
$
	D_{n,i}(t)=\E{ \chi( \cK_t( \cP(n) ) ) - \chi( \cK_t(  \cP'^{\,i}(n) )  ) \ | \ \cG_{n,i} }$.
We show that for each pair $(s,t)$ the random variable ${b_n^\prime}^{-1} \sum_{i=1}^{{b_n^\prime}} D_{n,i}(s) D_{n,i}(t)  $ attains the limit stated in the theorem. This is done in three steps. (i) We derive the covariance structure in the case where $\kappa \equiv 1$. (ii) We verify the claim if $\kappa$ is a blocked density of the form $\sum_{i=1}^{m^d} b_i \mathds{1}_{A_i}$. (iii) Finally, using the approximation result Theorem~\ref{T:ApproximationEuler}, we show the result for general density functions which satisfy \eqref{E:PropertyKappa}.

\textit{Step (i).} In this step, let $\kappa \equiv 1$ and let $s,t\in\R_+$. By construction $\cP(n)$ is the restriction of a homogeneous Poisson process to the cube $[-n^{1/d}/2, n^{1/d}/2 ]^d$. It is an immediate consequence of the strong stabilizing property outlined in Proposition~\ref{P:StrongStabilizingProperty} that there is an $N_0$ depending on $z$ and $T$ such that for all $n\ge N_0$ and $t\le T$
\[
	\chi( \cK_t( \cP(n) ) ) - \chi( \cK_t( \cP'^{\,z}(n) )  ) = \fD'_n(t,z) =: \fD_\infty(t,z) \quad a.s.
\]
for a certain random variable $\fD_\infty(t,z)$. Applying similar techniques as in the proof of Proposition 5.5 in \cite{krebs2019betti}, it can be shown that
\begin{align*}
		{b_n^\prime}^{-1} \sum_{i=1}^{{b_n^\prime}} D_{n,i}(s) D_{n,i}(t) &\to \E{ \E{ \fD_\infty(s,0) \ | \ \cF_0} \ \E{\fD_\infty(t,0) \ | \ \cF_0}	} \\
		& =: \gamma(s,t) \quad a.s. \text{ and in } L^1(P) \quad (n\to\infty). 
\end{align*}
Furthermore, the function $\gamma$ is continuous, this follows from the continuity of the variance function $\sigma^2(t) = \gamma(t,t)$ which we will prove now. 
Clearly, the variance function is bounded on each finite interval $[0,T]$ by the bounded moments condition from Lemma~\ref{L:BoundedMomentsCondition}. We have by construction 
\begin{align*}
		& |	\sigma^2(t) - \sigma^2(s) | = \big | \lim_{n\to\infty}  \E{ \ol \chi_n(t)^2 - \ol\chi_n(s)^2 } \big | \\
		&\le \sqrt{2} \ \limsup_{n\to\infty} \E{ (\ol\chi_n(t) - \ol\chi_n(s))^2 }^{1/2} \ \lim_{n\to\infty} \Big(\E{\ol\chi_n(t)^2}^{1/2} +  \E{\ol\chi_n(s)^2}^{1/2} \Big)\\
		&\le  \sqrt{2} \  \limsup_{n\to\infty} \E{ (\ol\chi_n(t) - \ol\chi_n(s))^2 }^{1/2}  \ (\sqrt{\sigma^2(t)} + \sqrt{\sigma^2(s)} ).
\end{align*}
Consider the difference	$\ol\chi_n(t) - \ol\chi_n(s)$ which can be written in terms of martingale differences as
\begin{align*}
	\E{ |\ol\chi_n(t) - \ol\chi_n(s)|^2 } &\le n^{-1} \sum_{i=1}^{{b_n^\prime}} \E{ | \chi(\cK_{t,n}) -\chi(\cK_{s,n})  - \chi(\cK'_{t,n,i})  + \chi(\cK'_{s,n,i}) |^2 }.
\end{align*}
Hence, it suffices to show that for a given offset $Q' = Q(z)^{(\delta)}$, for some $\delta>0$,
\begin{align}\begin{split}\label{E:MVN1}
		&\E{	\Big| \sum_{q=0}^\infty \# \{ \sigma\in \cK_{t,n}\setminus \cK'_{t,n,i} \ | \ \sigma\subset Q' , \dim \sigma = q, r(\sigma)\in (s,t] \} \Big|^2	} \\
		&\le C |s-t|
\end{split}\end{align}
for some constant $C$ which is independent of $z$ and $n$. To this end, denote the number of Poisson points in $Q'$ by $N$. Then $N$ is Poisson with parameter $\lambda$, say, that is bounded above by the Lebesgue measure of $Q'$ (because $\kappa = 1$). Let $\sigma_1$ be a $q_1$-simplex and $\sigma_2$ a $q_2$-simplex. One can show that
\[
		\p( r(\sigma_1) \in (s,t], r(\sigma_2) \in (s,t] )  \le \p( r(\sigma_1) \in (s,t] ) \le C q_1^p |s-t|
\]
for a certain $p\in\R_+$ which depends on the filtration (see Lemma~\ref{L:ContinuityVR} and Lemma~\ref{L:ContinuityCech}). Furthermore, we obtain for the factorial moment $\mathbb{E}[ N! / (N-q)! \1{N\ge q}] = \lambda^q$ for each $q\in\N_0$. Then, up to a multiplicative constant and using $\ell$ for the number of common points of the simplices $\sigma_1$ and $\sigma_2$, the left-hand side of \eqref{E:MVN1} is at most 
\begin{align*}
		& \sum_{q_2 = 0}^\infty \sum_{q_1=0}^{q_2}  \sum_{\ell=0}^{q_1+1} \E{ \binom{N}{\ell} \binom{N-\ell}{q_1+1-\ell} \binom{N-(q_1+1)}{q_2+1 - \ell}  \1{N \ge (q_1+1 + q_2+1 - \ell}  }q_1^p  |s-t| \\
		&=  \sum_{q_2 = 0}^\infty \sum_{q_1=0}^{q_2}  \sum_{\ell=0}^{q_1+1} \E{ \frac{N! \1{N \ge (q_1+1 + q_2+1 - \ell} }{(N - (q_1+1 + q_2+1 - \ell))!} } \frac{q_1^p }{\ell! (q_1+1-\ell)! (q_2+1-\ell)! } |s-t| \\
		&\le \sum_{\ell=0}^\infty \sum_{q_2=\ell-1}^{\infty}  \sum_{q_1=\ell-1}^{\infty}  \lambda^{ q_1+1 + q_2+1 - \ell }  \frac{q_1^p  }{\ell! (q_1+1-\ell)! (q_2+1-\ell)! } |s-t| \\
    		&= \sum_{\ell=0}^\infty \sum_{q_2=\ell-1}^{\infty}  \sum_{q_1=\ell-1}^{\infty} \frac{ \lambda^{q_1-(\ell-1)} q_1^p }{ (q_1 - (\ell-1) )! } \frac{ \lambda^{q_2-(\ell-1)} }{ (q_2 - (\ell-1) )! } \frac{ \lambda^\ell }{\ell!} |s-t| \le C |s-t|.
\end{align*}
This shows $| \sigma^2(t) - \sigma^2(s) | \le C \sqrt{|t-s|}$ for all $0\le s,t\le T$ and completes step (i).

\textit{Step (ii).} Let $\kappa$ be a blocked density of the form $\sum_{i=1}^{m^d} b_i \mathds{1}_{A_i}$ for positive numbers $b_i$ and a partition $(A_i)_{i=1}^{m^d}$ of $[0,1]^d$ into rectangular sets $A_i = \times_{i=1}^d I_{i,j_i}$, where the $(I_{i,j})_{j=1}^m$ partition $[0,1]$ in intervals of length $m^{-1}$. We have that
\[
		{b_n^\prime}^{-1} \sum_{i=1}^{{b_n^\prime}} D_{n,i}(s) D_{n,i}(t) \to \E{ \gamma( \kappa(Z)^{1/d}(s,t) ) } \quad a.s. \text{ and in } L^1(\p)
\]
for a random variable $Z$, which is distributed according to the blocked density $\kappa$. The calculations are similar to those in the proof of Proposition~5.7 in \cite{krebs2019betti}, we omit the details. This completes step (ii).

\textit{Step (iii).} Let $\epsilon>0$. Let $\nu$ be a blocked density function which approximates $\kappa$ uniformly such that $\| \nu - \kappa \|_\infty \le \epsilon$. Note that this is possible because $\kappa$ satisfies \eqref{E:PropertyKappa}. Using the result from Theorem~\ref{T:ApproximationEuler}, 
$
		\sup_{t\in [0,T] } \V( \ol \chi_{\kappa,n}(t) - \ol \chi_{\nu,n}(t) ) \le C \epsilon$,
for a certain constant $C\in\R_+$, which is independent of $n$. Hence, there is a constant $C'\in\R_+$, which is independent of $n$ such that
\[
		\sup_{s,t\in [0,T] } \big| \cov( \ol \chi_{\kappa,n}(s), \ol \chi_{\kappa,n}(t) ) - \cov( \ol \chi_{\nu,n}(s) , \ol \chi_{\nu,n}(t) ) \big| \le C \epsilon.
\]
Together with the continuity of $(s,t)\mapsto \gamma(s,t)$ and the results from Step (ii), this shows that ${b_n^\prime}^{-1} \sum_{i=1}^{{b_n^\prime}} D_{n,i}(s) D_{n,i}(t) \to \mathbb{E}[\gamma(\kappa(Z)^{1/d}(s,t)) ] $ $a.s.$ and in $L^1(\p)$, where $Z$ has density $\kappa$. In particular, $\V(\ol\chi_{\kappa,n}(t))\to \mathbb{E}[\gamma(\kappa(Z)^{1/d}(t,t)) ]$. So the calculations of step (3) are complete.

\textit{The binomial sampling scheme.} The result follows as in \cite{krebs2019betti}  using Poissonization arguments and the ideas of \cite{penrose2001central}; we only give a sketch and omit the technical details. (The arguments are very straightforward for the EC because the radius of stabilization is bounded, see also \cite{krebs2020law} for approximation results in the binomial sampling scheme.) 

Using the just cited sources, it is not difficult to show that for a general density $\kappa$ on $[0,1]^d$
\begin{align*}
		&\cov( \ol \chi_{\kappa,n}(s), \ol \chi_{\kappa,n}(t) ) \\
		&\to \E{ \gamma( \kappa(Z)^{1/d}(s,t) ) } - \E{ \mathbb{E}[ \Delta_\infty(s,\cP^h_{Z} ) ]}  \E{ \mathbb{E}[ \Delta_\infty(t,\cP^h_{Z} ) ] },
\end{align*}
where we extend the definition of $\Delta_\infty$ from \eqref{D:DeltaInf} to the homogeneous Poisson process on $\R^d$ with intensity $\tau>0$, which we denote by $\cP^h_{\tau}$ at this point. This is we formally replace $\cP$ by $\cP^h_\tau$ in \eqref{D:DeltaInf}; obviously the limit exists $a.s.$ by the same arguments. Then using the scaling properties of the \v Cech and Vietoris-Rips filtration, $\cK_{a t}  (P ) = \cK_t (aP)$ for any (finite) point cloud $P$ and filtration parameters $a,t>0$. Moreover, the distributions of the two homogeneous Poisson processes $ a \cP^h_\tau$ and $\cP^h_{a^{-d} \tau} $ coincide for all $\tau,a >0$. This implies
$ \mathbb{E}[ \Delta_\infty(t,\cP^h_\tau) ] =  \mathbb{E}[ \Delta_\infty(\tau^{1/d} t,\cP^h_1) ] $.
\end{proof}

\end{document}